\documentclass{amsart}

\usepackage{amsmath, amsthm}
\usepackage{eucal}

\usepackage{amssymb}
\usepackage{amscd}
\usepackage{palatino}
\usepackage{latexsym}
\usepackage{epsfig}
\usepackage{graphicx}
\usepackage{amsfonts}
\usepackage{psfrag}
\usepackage{youngtab, ytableau}
\usepackage{url}
\usepackage{cite}

\usepackage[margin=1in]{geometry}

\input xy
\xyoption{all}
\UseComputerModernTips

\numberwithin{equation}{section}
\numberwithin{figure}{section}

\newtheorem{theorem}{Theorem}[section]
\newtheorem{lemma}[theorem]{Lemma}
\newtheorem{proposition}[theorem]{Proposition}
\newtheorem{corollary}[theorem]{Corollary}

\newtheorem{remark}[theorem]{Remark}
\newtheorem{example}[theorem]{Example}
\newtheorem{question}[theorem]{Question}

\theoremstyle{definition}
\newtheorem{definition}[theorem]{Definition}

\newcommand{\C}{{\mathbb{C}}}

\newcommand{\Z}{{\mathbb{Z}}}
\newcommand{\Q}{{\mathbb{Q}}}

\newcommand{\into}{\hookrightarrow}

\usepackage{multicol}
\usepackage{color}
\definecolor{gold}{rgb}{0.85,.66,0}
\definecolor{cherry}{rgb}{0.9,.1,.2}
\definecolor{burgundy}{rgb}{0.8,.2,.2}
\definecolor{orangered}{rgb}{0.85,.3,0}
\definecolor{orange}{rgb}{0.85,.4,0}
\definecolor{olive}{rgb}{.45,.4,0}
\definecolor{lime}{rgb}{.6,.9,0}
\definecolor{green}{rgb}{.2,.7,0}
\definecolor{grey}{RGB}{180, 187, 198}
\definecolor{brown}{rgb}{.4,.3,.1}

\newcommand{\gl}{{\mathfrak gl}}
\renewcommand{\S}{S}

\DeclareMathOperator{\Hess}{Hess}
\newcommand{\Flags}{Flag}

\begin{document}

\title[Filtrations and regular nilpotent Hessenberg varieties]{A filtration on the cohomology rings of regular nilpotent Hessenberg varieties}

\author{Megumi Harada}
\address{Department of Mathematics and
Statistics\\ McMaster University\\ 1280 Main Street West\\ Hamilton, Ontario L8S4K1\\ Canada}
\email{Megumi.Harada@math.mcmaster.ca}
\urladdr{\url{http://www.math.mcmaster.ca/Megumi.Harada/}}

\author [T. Horiguchi]{Tatsuya Horiguchi}
\address{Department of Pure and Applied Mathematics\\
Graduate School of Information Science and Technology\\
Osaka University \\ 
1-5, Yamadaoka, Suita, Osaka, 565-0871 \\ Japan}
\email{tatsuya.horiguchi0103@gmail.com}

\author [S. Murai]{Satoshi Murai}
\address{Department of Mathematics\\
Faculty of Education\\
Waseda University\\
1-6-1 Nishi-Waseda, Shinjuku, Tokyo 169-8050 \\ Japan}
\email{s-murai@waseda.jp}

\author{Martha Precup}
\address{Department of Mathematics and Statistics\\ Washington University in St. Louis \\ One Brookings Drive \\ St. Louis, Missouri  63130 \\ U.S.A. }
\email{martha.precup@wustl.edu}
\urladdr{\url{https://www.math.wustl.edu/~precup/}}

\author [J. Tymoczko]{Julianna Tymoczko}
\address{Department of Mathematics \& Statistics \\ 
Clark Science Center \\ 
Smith College \\ 
Burton Hall 115 \\ 
Northampton, MA 01063 \\ 
U.S.A. }
\email{tymoczko@smith.edu}

\date{\today}

%%%%%%%%%%%%%%%%%%%%%
%  Abstract
%%%%%%%%%%%%%%%%%%%%%

\begin{abstract}

Let $n$ be a positive integer. The main result of this manuscript is a construction of a filtration 
on the cohomology ring of a regular nilpotent Hessenberg variety in $GL(n,\C)/B$ such that its associated graded ring has graded pieces (i.e., homogeneous components) isomorphic to rings which are related to the cohomology rings of Hessenberg varieties in $GL(n-1,\C)/B$, showing the inductive nature of these rings. 
In previous work, the first two authors, together with Abe and Masuda, gave an explicit presentation of these cohomology rings in terms of generators and relations.  We introduce a new set of polynomials which are closely related to the relations in the above presentation and obtain a sequence of equivalence relations they satisfy;  this allows us to derive our filtration.   In addition, we obtain the following three corollaries. First, we give an inductive formula for the Poincar\'e polynomial of these varieties. Second, we give an explicit monomial basis for the cohomology rings of regular nilpotent Hessenberg varieties with respect to the presentation mentioned above. Third, we derive a basis of the set of linear relations satisfied by the images of the Schubert classes in the cohomology rings of regular nilpotent Hessenberg varieties. Finally, our methods and results suggest many directions for future work; in particular, we propose a definition of ``Hessenberg Schubert polynomials'' in the context of regular nilpotent Hessenberg varieties, which generalize the classical Schubert polynomials. We also outline several open questions pertaining to them. 

\end{abstract}

\maketitle

\setcounter{tocdepth}{1}
\tableofcontents

%%%%%%%%%%%%%%%%%%%%%%%%%%%%%%%%%%
\section{Introduction}
\label{section:Introduction}
%%%%%%%%%%%%%%%%%%%%%%%%%%%%%%%%%%

Hessenberg varieties\footnote{Hessenberg varieties may be defined in more generality 
in other Lie types. In this manuscript, we focus on the Lie type A case, i.e. $G=GL(n,\C)$ (except in the introduction, where we mention some results for other Lie types).}
 are subvarieties of the full flag
variety $\Flags(\C^n)$ of nested sequences of linear subspaces in
$\C^n$. These varieties lie in a fruitful intersection of algebraic geometry, combinatorics, and representation theory, and they have been studied
extensively since the late 1980s, when they were first introduced by De Mari and Shayman and studied by De Mari, Procesi, and Shayman \cite{DeM, DeMShay, ma-pr-sh}.

We now describe the main result of this paper; the precise statement is given in Theorem~\ref{theorem: main}. 
Let $n$ be a positive integer. A function $h: \{1,2,\ldots,n\} \to \{1,2,\ldots,n\}$ is said to be a \textbf{Hessenberg function} if $h(i) \geq i$ for $1 \leq i \leq n$ and $h(i+1) \geq h(i)$ for $1 \leq i \leq n-1$. For $A$ an $n \times n$ matrix with complex entries and $h$ a Hessenberg function, the \textbf{Hessenberg variety} determined by $A$ and $h$, denoted $\Hess(A,h)$, is a subvariety of the full flag variety $\Flags(\C^n)$. (Details are in Section~\ref{section:Preliminary}.) When the matrix is chosen to be a regular nilpotent matrix $\mathsf{N}$, i.e.  
a matrix whose Jordan form consists of exactly one Jordan block with
corresponding eigenvalue equal to $0$, then the corresponding variety is called a \textbf{regular nilpotent Hessenberg variety}. The cohomology rings
\footnote{In this manuscript, unless stated otherwise, we work with singular cohomology with coefficients in $\Q$.} 
of these varieties are the main objects of study in this manuscript. 
Now let $h: \{1,2,\ldots,n\} \to \{1,2,\ldots,n\}$ be a Hessenberg function and let $\mathsf{N}$ be a regular nilpotent matrix.
Our Theorem~\ref{theorem: main} shows that there exists a positive integer $p \leq n$, a sequence of rings $\mathcal{A}_s^{h}$ for $1 \leq s \leq p$ such that $\mathcal{A}_1^{h} \cong H^*(\Hess(\mathsf{N},h))$, and a filtration of $H^*(\Hess(\mathsf{N},h))$ as a vector space; here, by a filtration we mean that there is a sequence of maps 
\[
0 \to \mathcal{A}_p^{h} \hookrightarrow \mathcal{A}_{p-1}^{h} \hookrightarrow \cdots \hookrightarrow \mathcal{A}_2^{h} \hookrightarrow \mathcal{A}_1^{h} \cong H^*(\Hess(\mathsf{N},h))
\]
where each arrow is an inclusion of vector spaces. Furthermore, for each $s$ with $1\leq s \leq p$, we can define a ``smaller'' Hessenberg function $h^{(s)}: \{1,2,\ldots, n-1\} \to \{1,2,\ldots,n-1\}$ defined on $\{1,2,\ldots,n-1\}$ instead of $\{1,2,\ldots,n\}$ which, intuitively, deletes the $s$-th row and the $s$-th column of the diagram of $h$. (A precise definition is given in~\eqref{eq: def hs}.) In addition to the existence of the filtration above, our main result also shows that the successive quotients $\mathcal{A}_s^{h}/\mathcal{A}_{s+1}^{h}$ are isomorphic to an analogously defined ring associated to the smaller Hessenberg function $h^{(s)}$. More precisely, we prove that 
\[
\mathcal{A}_s^{h}/\mathcal{A}_{s+1}^{h} \cong \mathcal{A}_{r_s}^{h^{(s)}}
\]
for an appropriately chosen integer $r_s$. Our result therefore opens the door for an inductive analysis of these rings, which include the cohomology rings of regular nilpotent Hessenberg varieties.

Some remarks are in order. Building on work of \cite{FHM, HHM}, the first two authors and H.~Abe and Masuda gave a ring presentation of the cohomology rings 
of regular nilpotent Hessenberg varieties in the paper \cite{AHHM}, as a polynomial ring $\Q[x_1,\ldots,x_n]$ modulo an ideal $I_h$, with an explicit list of $n$ generators which we notate as $f_{h(1),1},\ldots,f_{h(n),n}$. This explicit presentation, stated precisely in Theorem~\ref{theorem: cohomology}, is what allows us to make the arguments in the present paper. In another direction, the second and third authors and T.~Abe, Masuda, and Sato showed in \cite{AHMMS} that the cohomology rings of regular nilpotent Hessenberg varieties can be described in the language of hyperplane arrangements. 
From this point of view, the problem of giving an explicit ring presentation of $H^*(\Hess(\mathsf{N},h))$ was recently solved in all Lie types \cite{AHMMS, EHNT19a}.
The main new insight in the current paper, which allows us to make the relevant arguments, is the introduction of a set of polynomials $g_{h(1),1}, g_{h(2),2}, \ldots, g_{h(n),n}$ which are closely related to the above generators $f_{h(i),i}$ of the ideal $I_h$. The key technical points which form the core of our proof are that these new polynomials satisfy the same recursive formulas as the ideal generators (Lemma~\ref{lemma: recursion}), and moreover, that the ideal generators inductively satisfy a very simple algebraic equation involving the $g_{h(j),j}$ (Lemma~\ref{lemma:4-1}). We also heavily use the fact, proven in \cite{AHHM}, that the sequence $\{f_{h(1),1}, f_{h(2),2}, \cdots, f_{h(n),n}\}$ is a regular sequence.

This paper also gives several immediate applications of our main Theorem~\ref{theorem: main}. The first is an inductive formula for the Poincar\'e polynomials of regular nilpotent Hessenberg varieties, stated in Corollary~\ref{corollary: inductive Poincare}. The Betti numbers of these varieties had been computed by previous work of the last two authors \cite{Precup, Tymoczko} and a closed formula for the Poincar\'e polynomial, was given in \cite{SommersTymoczko, AHMMS}. 
Therefore, our formula is certainly not the first formula for the Poincar\'e polynomial of $\Hess(\mathsf{N},h)$, but it reveals new inductive properties of this Poincar\'e polynomial. 
As a second application, we give a set of monomials in the polynomial ring $\Q[x_1,\ldots,x_n]$ such that their images under the quotient map $\Q[x_1,\ldots,x_n] \to \Q[x_1,\ldots,x_n]/I_h \cong H^*(\Hess(\mathsf{N},h))$ form a $\Q$-vector space basis of $H^*(\Hess(\mathsf{N},h))$; the result is stated precisely in Theorem~\ref{theorem: monomial basis}. These monomials are natural from the point of view of Schubert calculus, and in fact were the primary motivation for this manuscript. Note that the result of \cite{EHNT19} gives a basis for the same ring whose elements are monomials in the positive roots, whereas the basis given in Theorem~\ref{theorem: monomial basis} is formed by monomials in the original variables $x_i$. 
Finally, as a third application of our techniques as well as the fact that the cohomology of the full flag variety surjects onto $H^*(\Hess(\mathsf{N},h))$, we give an algorithm for deriving a basis for the set of linear relations satisfied by the images of the Schubert classes in $H^*(\Hess(\mathsf{N},h))$.

Our methods and results suggest many possible avenues of future work. 
Firstly, as mentioned above, our monomial basis has natural interpretations in terms of Schubert calculus, and we use this basis to propose a definition of ``Hessenberg Schubert polynomials'', in the context of regular nilpotent Hessenberg varieties, in Section~\ref{sec: HS polynomials}. 
Secondly, it is known by work of the first two authors and Abe and Masuda that the cohomology ring $H^*(\Hess(\mathsf{N},h))$ is isomorphic to the $\S_n$-invariant subring of the ``dot action'' representation (defined by the fifth author in \cite{Tymoczko-rep}) on the cohomology ring $H^*(\Hess(\mathsf{S},h))^{\S_n}$ of the corresponding regular semisimple  Hessenberg variety (here $\mathsf{S}$ is diagonalizable, with distinct eigenvalues) \cite[Theorem B]{AHHM}. Thus it is natural to ask whether a filtration which is analogous to ours, and which is compatible with the dot action, also exists for $H^*(\Hess(\mathsf{S},h))$. 
Thirdly, on a related note, it would be interesting to see a \emph{geometric} interpretation of our filtration and the isomorphisms of the graded pieces $\mathcal{A}_s^{h}/\mathcal{A}_{s+1}^{h}$ with analogous rings for smaller Hessenberg functions. 
Fourthly, the polynomials $g_{h(j),j}$ which we introduce are defined by using the same recursive formula as the $f_{h(j),j}$, but with a different set of initial values. This suggests that it may be profitable to study the entire family of polynomials defined by these recursive formulas, with differing initial values.

\bigskip
\noindent \textbf{Acknowledgements.}   
We are grateful to the hospitality of the Mathematical Sciences Research Institute in Berkeley, California, and the Osaka City University Advanced Mathematical Institute in Osaka, Japan, where parts of this research was conducted. 
Some of the material contained in this paper are based upon work supported by the National Security Agency
under Grant No. H98230-19-1-0119, The Lyda Hill Foundation, The McGovern
Foundation, and Microsoft Research, while the first, fourth, and fifth authors were in residence at the
Mathematical Sciences Research Institute in Berkeley, California, during the summer of 2019.
We are also grateful for a crucial idea from Claudia Miller and for helpful conversations with Adam Van Tuyl. 
The first author is supported by a Natural Science and Engineering Research Council Discovery Grant and a Canada Research Chair (Tier 2) from the Government of Canada. The second author is supported by
JSPS Grant-in-Aid for JSPS Research Fellow:
17J04330 and by JSPS Grant-in-Aid for Young Scientists:
19K14508. 
The third author is partially supported by Kakenhi 16J04761 and Waseda
University Grant Research Base Creation 2019C-134.
The fifth author is supported in part by NSF DMS-1800773.

%%%%%%%%%%%%%%%%%%%%%%%%%%%%%%%%%%
\section{Preliminaries}
\label{section:Preliminary}
%%%%%%%%%%%%%%%%%%%%%%%%%%%%%%%%%%

Let $n$ be a positive integer. Throughout, we will use the notation $[n] := \{1,2,\ldots,n\}$. 

\begin{definition}\label{definition:Hessenberg function} 
A \textbf{Hessenberg function} is a function $h: [n] \to [n]$ satisfying the following two conditions
\begin{enumerate}
\item $h(i) \geq i$ for all $1 \leq i \leq n$, and 
\item $h(i+1) \geq h(i)$ for all $1 \leq i < n$. 
\end{enumerate}
We frequently write a Hessenberg function by listing its values in sequence,
i.e. $h = (h(1), h(2), \ldots, h(n))$.
\end{definition}

We now define the objects which are the geometric motivation for this manuscript. 
Let $h:[n]\to[n]$ be a Hessenberg function and let 
$A$ be an $n\times n$ matrix in $\gl(n,\C)$. Then the
\textbf{Hessenberg variety} $\Hess(A,h)$ associated to $h$ and $A$ is
defined to be  
\begin{align}\label{eq:def-general Hessenberg}
\Hess(A,h) := \{ V_{\bullet}  \in \Flags(\C^n) \;
\vert \;  AV_i \subset 
V_{h(i)} \text{ for all } i=1,\ldots,n\} \subset  \Flags(\C^n).
\end{align}
In particular, by definition $\Hess(A,h)$ is a subvariety of
$\Flags(\C^n)$, and if $h=(n,n,\ldots,n)$, then it is immediate
from~\eqref{eq:def-general Hessenberg}
that $\Hess(A,h)=\Flags(\C^n)$ for any choice of
$A$. Thus the full flag variety $\Flags(\C^n)$ is itself a special
case of a Hessenberg variety.

Let $\mathsf{N}$ denote a regular
nilpotent matrix in $\gl(n,\C)$, i.e., 
a matrix whose Jordan form consists of exactly one Jordan block with
corresponding eigenvalue equal to $0$. 
Then, for any choice of
Hessenberg function $h$, we call $\Hess(\mathsf{N},h)$ the 
\textbf{regular nilpotent Hessenberg variety} (associated to $h$).

We now define the polynomials which are the main focus of this manuscript.  
For $1 \leq j \leq i \leq n$ we define a polynomial 
\begin{equation}\label{eq: def fij} 
f_{i,j}(x_1, \ldots, x_n) :=\sum_{k=1}^j \left( \prod_{\ell=j+1}^i (x_k-x_\ell) \right) x_k.
\end{equation}
Here we take the convention that if $j+1 > i$ (i.e. we are in the situation $i=j$) then the product appearing in the RHS is $1$. Thus for any $i$ with $1 \leq i \leq n$ we have 
\begin{equation}\label{eq: def fii}
f_{i,i}(x_1,\ldots, x_n) := x_1+x_2+\cdots+x_i.
\end{equation} 
We also define the polynomials
\begin{equation}\label{eq: def gij} 
g_{i,j}(x_1, \ldots,x_n) :=\sum_{k=1}^j \left( \prod_{\ell=j+1}^i (x_k-x_\ell) \right)
\end{equation}
for $1 \leq j \leq i \leq n$. In the case when $j+1>i$ (i.e. when $i=j$) we take the same convention as for the $f_{i,j}$ and we see that for any $i$ with $1 \leq i \leq n$ we have 
\begin{equation}\label{eq: def gii}
g_{i,i}(x_1,\ldots, x_n) := \sum_{k=1}^i 1 = i. 
\end{equation}
As can be seen from the definitions, the polynomials $f_{i,j}$ and $g_{i,j}$ are intimately related, and we will use this to our advantage in what follows. 

\begin{remark}\label{remark: f and g homogeneous} 
It is immediate from the definitions that for $1 \leq j \leq i \leq n$, the polynomial $f_{i,j}$ is homogeneous of degree $i-j+1$, and $g_{i,j}$ is homogeneous of degree $i-j$. 
\end{remark} 

It is useful to visualize both $f_{i,j}$ and $g_{i,j}$ as corresponding to the $(i,j)$-th matrix entry in an $n\times n$ matrix as follows: 
\begin{equation}\label{eq:matrix fij}
\begin{pmatrix}
f_{1,1} & 0 & \cdots & \cdots & 0 \\
f_{2,1} & f_{2,2} & 0 & \cdots & 0 \\
f_{3,1} & f_{3,2} & f_{3,3} & \ddots &\vdots \\
\vdots &        &                 &  \ddots &  \\
f_{n,1} & f_{n,2} & \cdots &  & f_{n,n} 
\end{pmatrix}
\end{equation}
and similarly for the $g_{i,j}$. Visualized in this manner, the equations~\eqref{eq: def fii} and~\eqref{eq: def gii} give simple explicit formulas for the entries along the main diagonal, i.e. the $(i,i)$-th entries. 
The polynomials $f_{i,j}$ were originally defined in \cite{AHHM} by the recursive formula \eqref{eq: fij recursion} below. Motivated by this, we have the following lemma.

\begin{lemma}\label{lemma: recursion} 
For $(i,j)$ with $1 \leq j < i \leq n$ we have 
\begin{equation} \label{eq: fij recursion}
f_{i,j}=f_{i-1,j-1}+\big(x_j-x_i\big)f_{i-1,j}
\end{equation}
where we take the convention $f_{i,0}:=0$ for any $i$. 
We also have 
\begin{equation}\label{eq: gij recursion} 
g_{i,j} = g_{i-1,j-1} + (x_j - x_i) g_{i-1,j}
\end{equation}
where again we take the convention that $g_{i,0} = 0$ for any $i$.

\end{lemma} 

\begin{proof} 
For the recursive formula \eqref{eq: fij recursion} for the $f$'s, see \cite[Lemma~6.5]{AHHM}. We show the recursive formula \eqref{eq: gij recursion} for the $g$'s.
First consider the case when $j=1$. Then $g_{i-1,j-1} = g_{i-1, 0} = 0$, so~\eqref{eq: gij recursion} 
becomes 
\[
g_{i,1} = (x_1 - x_i) g_{i-1,1}.
\]
Since $i>j$ by assumption, we know $i\geq 2$. By definition of the $g_{i,j}$ we have 
\[
g_{i,1} = \prod_{\ell=2}^i (x_1 - x_\ell) = (x_1 - x_i) \left( \prod_{\ell=2}^{i-1}(x_1 - x_\ell) \right)  
= (x_1-x_i)  g_{i-1,1} 
\]
as desired, where the product is taken to be equal to $1$ if the top index is less than the bottom index.

Now we consider the case $j \geq 2$. 
We start with the RHS of~\eqref{eq: gij recursion} and compute: 
\begin{equation*}
\begin{split}
& g_{i-1,j-1} + (x_j-x_i) g_{i-1,j} \\
& = \sum_{k=1}^{j-1} \left( \prod_{\ell=j}^{i-1} (x_k-x_\ell) \right) + 
(x_j - x_i) \left( \sum_{k=1}^j \left( \prod_{\ell=j+1}^{i-1} (x_k - x_\ell) \right) \right) \textup{ by definition } \\
 & = \sum_{k=1}^{j-1} \left( \prod_{\ell=j}^{i-1} (x_k-x_\ell) \right) + (x_j - x_i) \left( \sum_{k=1}^{j-1} \left( \prod_{\ell=j+1}^{i-1} (x_k - x_\ell) \right) \right)
 + (x_j - x_i) \left( \prod_{\ell=j+1}^{i-1} (x_j-x_\ell) \right) \\
 & = \sum_{k=1}^{j-1} \left( \prod_{\ell=j}^{i-1} (x_k-x_\ell)\right) + (x_j - x_i) \left( \sum_{k=1}^{j-1} \left( \prod_{\ell=j+1}^{i-1} (x_k - x_\ell) \right) \right)
 +  \left( \prod_{\ell=j+1}^{i} (x_j-x_\ell) \right) \\
 & = \sum_{k=1}^{j-1} \left( \left( \prod_{\ell=j+1}^{i-1} (x_k-x_\ell)\right)  \left( (x_k - x_j) + (x_j - x_i) \right) \right) + 
 \left( \prod_{\ell=j+1}^i (x_j - x_\ell) \right) \\
 & = \sum_{k=1}^{j-1} \left( \prod_{\ell=j+1}^i (x_k-x_\ell) \right)  + \left( \prod_{\ell=j+1}^i (x_j - x_\ell)\right) \\
 & = \sum_{k=1}^j \left( \prod_{\ell=j+1}^i (x_k - x_\ell)\right) \\
 & = g_{i,j}
\end{split}
\end{equation*}
so the RHS equals the LHS and we are done.
\end{proof} 

Let $\mathcal{R} :=\Q[x_1,\ldots,x_n]$ denote the polynomial ring in variables $x_1, \ldots, x_n$, and define the ideal $I_h$ associated with a Hessenberg function $h$ by
\begin{equation}\label{eq: def Ih} 
I_h :=\langle f_{h(1),1}, f_{h(2),2}, \ldots, f_{h(n),n} \rangle.
\end{equation} 

Then the quotient ring $\mathcal{R}/I_h$ has the following geometric meaning as 
shown in \cite{AHHM}. 

\begin{theorem} \cite[Theorem~A]{AHHM} \label{theorem: cohomology} 
  Let $n$ be a positive integer. Let $\mathsf{N}$
  denote a regular nilpotent matrix in $\mathfrak{gl}(n,\C)$,  $h:[n] \to
  [n]$ a Hessenberg function, and let
 $\Hess(\mathsf{N},h) \subset  \Flags(\C^n)$ be
  the associated regular nilpotent Hessenberg variety. Then the
  restriction map
\[
H^*(\Flags(\C^n)) \to H^*(\Hess(\mathsf{N},h))
\]
is surjective, and there is an isomorphism of graded $\Q$-algebras
\begin{equation}\label{eq:intro Theorem A} 
H^*(\Hess(\mathsf{N},h)) \cong \Q[x_1, \ldots, x_n]/I_h = \mathcal{R}/I_h 
\end{equation}
where $I_h$ is the ideal of $\mathcal{R} = \Q[x_1, \ldots, x_n]$ defined in~\eqref{eq: def Ih}.
Here, the isomorphism doubles the grading on the right hand side, namely $\deg(x_i)=2$ for all $1 \leq i \leq n$. 
\end{theorem}

\begin{remark} \label{remark:CohomologyFlag}
The cohomology ring $H^*(\Flags(\C^n);\Z)$ has the well-known Borel presentation 
\begin{equation}\label{eq: Borel} 
H^*(\Flags(\C^n);\Z) \cong \Z[x_1,\ldots,x_n]/\langle e_1,\ldots, e_n\rangle 
\end{equation}
where $e_i$ denotes the $i$-th elementary symmetric polynomial for each $i$. 
The $x_i$ appearing in \eqref{eq: Borel} represents the negative of the first Chern class of the $i$-th tautological line bundle over $\Flags(\C^n)$. 
The $x_i$ in~\eqref{eq:intro Theorem A} denotes the image of the $x_i$ in~\eqref{eq: Borel} under the restriction map $H^*(\Flags(\C^n)) \to H^*(\Hess(\mathsf{N},h))$.
\end{remark}

Next, we define a \emph{family} of rings which include the cohomology rings appearing in Theorem~\ref{theorem: cohomology} 
as special cases. 
Let $h: [n] \to [n]$ denote a Hessenberg function and let $1 \leq s \leq n+1$. We define 
\begin{equation}\label{eq: def Ajh} 
\mathcal{A}_s^{h} :=\mathcal{R}/\langle g_{h(1),1},\ldots,g_{h(s-1),s-1},f_{h(s),s},\ldots,f_{h(n),n} \rangle.
\end{equation}  
Thus the ring $\mathcal{A}_s^{h}$ is obtained by replacing the first $s-1$ relations $f_{h(1),1}, \ldots, f_{h(s-1),s-1}$ defining the ideal $I_h$ with the polynomials $g_{h(1),1}, \ldots, g_{h(s-1), s-1}$ respectively. 
In particular, it is clear that $\mathcal{A}_1^{h}=\mathcal{R}/I_h \cong H^*(\Hess(\mathsf{N},h))$, so these rings generalize the cohomology rings appearing in the above theorem. 
Note that $\mathcal{A}_{n+1}^{h}$ is the zero ring because $g_{h(n),n}=g_{n,n}=n$ is a constant.

Our arguments will depend heavily on techniques from commutative algebra, stemming from the fact that the sequences appearing in the definition of $\mathcal{A}_s^{h}$ above are regular sequences.  
We first recall the definition of a regular sequence from commutative algebra.

\begin{definition}\label{def of regular seq for Hilb series}
For a ring $S$, a sequence $f_1,\dots,f_r\in S$ is called a
\textbf{regular sequence} if: 
\begin{itemize}
\item[(i)] $f_i$ is non-zero, and not a zero-divisor, in $S/(f_1,\dots,f_{i-1})$ for $i=1,\dots,r$, and 
\item[(ii)] $S/(f_1,\dots,f_r)\neq0$.
\end{itemize}
\end{definition}

If $S$ is a graded $\Q$-algebra and $f_1, \dots, f_r$ are \emph{positive-degree, homogeneous} elements, 
then it is well-known (cf.~for example \cite[Chapter 1, Section 5.6]{stan96}) that $\{f_1, \dots, f_r\}$ is a regular sequence if and only if the set $\{f_1,\dots,f_r\}$ is algebraically
independent over $\Q$ and $S$ is a free
$\Q[f_1,\dots,f_r]$-module. 
The following is then immediate. 

\begin{lemma}\label{lemma: perm} 
Let $\mathcal{R}=\Q[x_1,\ldots,x_n]$ be a polynomial ring and let $g_1,\ldots,g_n \in \mathcal{R}$ be a regular sequence of 
positive-degree homogeneous polynomials in $\mathcal{R}$. If $\{g_1,\ldots,g_n\} \subseteq \mathcal{R}$ is a regular sequence, 
then for any permutation $\sigma \in \S_n$, the reordering $\{g_{\sigma(1)}, \ldots, g_{\sigma(n)}\}$ is also a regular sequence. 
\end{lemma} 

\begin{proof} 
As remarked before the statement of the lemma, it is known that if $g_1, \ldots, g_n$ are positive-degree and homogeneous, then the property of being a regular sequence can be characterized by algebraic independence over $\Q$ and the freeness of $\mathcal{R}$ as a $\Q[g_1,\ldots, g_n]$-module. Both conditions depend only on the set $\{g_1,\ldots,g_n\}$ and are independent of their ordering.  Hence the statement of the lemma follows. 
\end{proof}

The following lemma will be fundamental to our arguments below.

\begin{lemma} \label{lemma:key}
Let $\mathcal{R}=\Q[x_1,\ldots,x_n]$ be a polynomial ring and let $\{g_1,\ldots,g_n\} \subset \mathcal{R}$ be a regular sequence of positive-degree homogeneous polynomials in $\mathcal{R}$.
Assume that $g_n=g'_n \cdot g''_n$ for some positive-degree homogeneous polynomials $g'_n, g''_n \in \mathcal{R}$.
Then, the linear map induced by multiplication by $g''_n$, i.e., the map 
\begin{equation}\label{eq: key injection}
\times g''_n: \mathcal{R}/\langle g_1,\ldots,g_{n-1},g'_n \rangle \to \mathcal{R}/ \langle g_1,\ldots,g_{n-1},g_n\rangle, \quad [f] \mapsto [ f \cdot g''_n] 
\end{equation}
is well-defined and injective. In fact, the above map fits into an exact sequence of $\mathcal{R}$-modules  
\begin{equation}\label{eq: key exact seq}
0 \to \mathcal{R}/\langle g_1,\ldots,g_{n-1},g'_n \rangle \to \mathcal{R}/ \langle g_1,\ldots,g_{n-1},g_n\rangle \to 
\mathcal{R}/ \langle g_1, \ldots, g_{n-1}, g_n, g''_n \rangle \to 0.
\end{equation}
\end{lemma}

\begin{proof} 
To check well-definedness, it suffices to check that the generators $\{g_1, g_2,\ldots, g_n'\}$ of the ideal on the LHS of~\eqref{eq: key injection} go to the ideal generated by $\{g_1,g_2,\ldots,g_n\}$, and this is clear. For the injectivity, it is useful to notice that we can describe the map~\eqref{eq: key injection} also as 
\[
\times g''_n: \left(\mathcal{R}/\langle g_1,\ldots, g_{n-1}\rangle \right)/\langle g'_n \rangle \to \left( \mathcal{R}/\langle g_1,\ldots,g_{n-1}\rangle \right)/\langle g_n \rangle,  \quad [f] \mapsto [f \cdot g''_n],
\]
where by slight abuse of notation we also denote by $f$ an element of $S := \mathcal{R}/\langle g_1,\ldots,g_{n-1}\rangle$. By assumption, $\{g_1,\ldots,g_n\}$ is a regular sequence, so the image of $g_n \in \mathcal{R}$ in $S$ is not a zero divisor in $S$. It follows that $g'_n, g''_n$ are also not zero divisors in $S$. Thus to prove the claim it is enough to show the following: for $S$ a commutative ring and $b$ not a zero divisor in $S$, the map 
\[
\times b: S/\langle a \rangle \to S/\langle ab \rangle 
\]
is injective, for any $a \in S$. To see this, suppose $f \in S$ and $fb \in \langle ab\rangle$. We wish to show $f \in \langle a\rangle$. But if $fb \in \langle ab\rangle$ then there exists $g \in S$ with $fb = gab$, so $(f-ga)b=0$ in $S$. Since $b$ is not a zero divisor in $S$, we have $f-ga=0$ in $S$, so $f \in \langle a\rangle$ as desired. 
To see the exactness of the sequence~\eqref{eq: key exact seq}, the only substantive point remaining is exactness in the middle. Note that the image of the map $\times g''_n$ is the (image under the natural quotient map of the) space 
\[
\{ g''_n \cdot f \, \mid \, f \in \mathcal{R}\} 
\]
and the kernel of the projection map $\mathcal{R}/ \langle g_1,\ldots,g_{n-1},g_n\rangle \to 
\mathcal{R}/ \langle g_1, \ldots, g_{n-1}, g_n, g''_n \rangle$ is also the (image under the natural quotient map of the) space
\[
\{g''_n \cdot f \, \mid \, f \in \mathcal{R}\} 
\]
so they are equal.  
\end{proof}

Finally, we will need the following two facts. The first says that if an element in a regular sequence can be factored, then we can ``take off'' one of the factors and still have a regular sequence. This is immediate from the definition of regular sequences so we omit the proof. 

\begin{lemma}\label{lemma: peel} 
Let $\mathcal{R}=\Q[x_1,\ldots,x_n]$ be a polynomial ring and $\{g_1,\ldots,g_n\} \subset \mathcal{R}$ a regular sequence.
Assume that $g_n=g'_n \cdot g''_n$ for some positive-degree homogeneous polynomials $g'_n, g''_n \in \mathcal{R}$.
Then $\{g_1, \ldots, g_{n-1}, g'_n\}$ is also a regular sequence. 
\end{lemma}

The second statement is the following well-known characterization of regular sequences phrased in terms of Hilbert series \cite[p.35]{stan96}. For any positively graded algebra $\mathcal{A} = \oplus_{k \geq 0} \mathcal{A}_k$ with finite-dimensional graded components, we let $F(\mathcal{A},t)$ denote its Hilbert series, defined as 
\[
F(\mathcal{A},t) := \sum_{k \geq 0} \dim_\Q(\mathcal{A}_k) t^k.
\]

\begin{proposition}\label{prop: regular seq in terms of Hilbert} 
Let $\mathcal{R} = \Q[x_1,\dots,x_n]$ and let $\{g_1,\dots, g_n\} \in \mathcal{R}$ 
be a set of positive-degree homogeneous polynomials. 
Then $\{g_1,\ldots, g_n\}$ is a regular sequence if and only if 
\begin{align*}
F(\Q[x_1,\dots,x_n]/\langle g_1,\dots,g_n\rangle ,t)
=\prod_{k=1}^n (1+t+t^2 + \cdots + t^{\deg(g_k)-1}) 
\end{align*}
where $\deg(g_k)$ is the degree of $g_k$ in $\Q[x_1,\dots,x_n]$. 
\end{proposition}

%%%%%%%%%%%%%%%%%%%%%%%%%%%%%%%%%%
\section{Statement of the main theorem}
\label{section:Main theorem}
%%%%%%%%%%%%%%%%%%%%%%%%%%%%%%%%%%

We now give a precise statement of our main theorem. 
To do so, we need a bit more notation. Let $h:[n] \to [n]$ be a Hessenberg function. 
The \textbf{diagram of $h$} is the $n \times n$ array in which a 
``star'' is drawn in the box $(i,j)$ if $i \leq h(j)$. 

\begin{example}\label{example: stars}
Let $h=(2,4,4,5,6,6)$. The diagram of $h$ is given below. 
\[\ytableausetup{centertableaux} 
h:  \; \;   \begin{ytableau}
\star &  \star & \star & \star& \star& \star \\ 
\star & \star &\star & \star & \star & \star  \\
 & \star  & \star & \star& \star& \star \\ 
 & \star & \star & \star & \star & \star \\ 
 & &  & \star& \star& \star \\ 
  & & &  & \star & \star \\ 
\end{ytableau} 
\]
\end{example} 

In addition, for $s$ such that $1 \leq s \leq n$, we let $r_s$ denote the minimal number $m$ such that $h(m) \geq s$. 
Namely, 
\begin{equation}\label{eq: def rs}
r_s := \min \{ m \, \mid \, h(m) \geq s \}.
\end{equation}

\begin{example}
Continuing Example~\ref{example: stars},  
it is not hard to check that for $h=(2,4,4,5,6,6)$, we have $r_1 = 1, r_2=1, r_3=2, r_4 = 2, r_5 = 4$ and $r_6=5$. 
\end{example}

The following is straightforward. 
\begin{lemma}\label{lemma: rs} 
For any Hessenberg function $h: [n] \to [n]$ and any $s$, $1 \leq s \leq n$, we have $r_s \leq s$. 
\end{lemma} 

\begin{proof} 
By definition of Hessenberg functions, we must have $h(s) \geq s$, so $s$ is contained in the set appearing in~\eqref{eq: def rs}.
\end{proof} 

Let $h^{(s)}$ be the Hessenberg function obtained from $h$ by removing $s$-th row and $s$-th column from the diagram of $h$ and interpreting the result as the diagram of a Hessenberg function from $[n-1]$ to $[n-1]$. 
We give an example to illustrate.

\begin{example} 
Continuing the above example, let $h=(2,4,4,5,6,6)$ and let $s=2$. If we delete the $2$nd row and $2$nd column from the diagram of $h$, pictorially indicated by the shaded boxes in the figure, then what remains represents the Hessenberg function $h^{(2)} = (1,3,4,5,5)$. 
        \[\ytableausetup{centertableaux} 
        h:  \; \;   \begin{ytableau}
        \star & *(grey) \star & \star & \star& \star& \star \\ 
        *(grey)\star & *(grey)\star &*(grey)\star &*(grey)\star &*(grey)\star &*(grey)\star  \\
         & *(grey)\star  & \star & \star& \star& \star \\ 
         & *(grey)\star & \star & \star & \star & \star \\ 
         & *(grey) &  & \star& \star& \star \\ 
          & *(grey) & &  & \star & \star \\ 
        \end{ytableau}
        \quad \quad \quad 
        h^{(2)}: \; \; 
        \begin{ytableau} 
        \star & \star & \star & \star & \star \\ 
         & \star & \star & \star & \star \\ 
          & \star & \star & \star & \star \\ 
           &  & \star & \star& \star \\ 
            &  & & \star & \star 
        \end{ytableau} 
        \]
\end{example} 

We can also describe $h^{(s)}:[n-1] \to [n-1]$ explicitly as follows:  
\begin{equation}\label{eq: def hs}
h^{(s)}(m)=\begin{cases}h(m) & {\rm if} \ 1 \leq m \leq r_s-1, \\
h(m)-1 & {\rm if} \ r_s \leq m \leq s-1, \\
h(m+1)-1 & {\rm if} \ s \leq m \leq n-1
\end{cases}
\end{equation} 
as the reader may check.

Finally, we define the number $\mathsf{p}(h)$ to be the first position at which the diagram of $h$ hits the main diagonal. More precisely, we define 
\begin{equation}\label{eq: def ph}
\mathsf{p}(h) := \min\{ m \in [n] \, \mid \, h(m)=m \}.
\end{equation} 
Note that any Hessenberg function $h$ satisfies $h(n)=n$, so the set $\{m \in [n]: h(m)=m\}$ is non-empty and the integer $\mathsf{p}(h)$ is well-defined. 
As we already noted in~\eqref{eq: def gii}, the polynomial $g_{m,m}$ is a constant polynomial for any $m \in [n]$. Thus, if $h(m)=m$, then $\mathcal{A}^{h}_{s}$ for any $s>m$ is the zero ring (since the ideal by which we quotient includes $g_{m,m}$, a constant). We need to avoid these degenerate cases in the arguments below. 

\begin{remark} 
We say that a Hessenberg function is \textbf{connected} if $h(i) >i$ for all $1 \leq i < n$. In this case, $\mathsf{p}(h)=n$. The reason for this terminology is that, in this case, the corresponding regular semisimple Hessenberg variety is connected in the classical topological sense. In many situations, it is natural to restrict attention to connected Hessenberg functions. However, we cannot place this assumption on our Hessenberg functions because we also need to deal with disconnected
 (i.e., \ not connected) Hessenberg functions in our inductive argument. 
\end{remark}

We can now state the main theorem.

\begin{theorem} \label{theorem: main} 
Let $n$ be a positive integer and $h: [n] \to [n]$ a Hessenberg function. For each $s$, $1 \leq s \leq \mathsf{p}(h)$, there is a well-defined, injective linear map 
\begin{equation}\label{eq: mult xs}
\xymatrix{
\mathcal{A}^{h}_{s+1} \ar@<-0.3ex>@{^{(}->}[r]^-{\times x_s} & \mathcal{A}_s^{h} 
}
\end{equation}
induced by ``multiplication by $x_s$'', giving rise to a filtration 
\begin{equation}\label{eq: filtration main}
\xymatrix {
0  \ar@<-0.3ex>@{^{(}->}[r] & \mathcal{A}_{\mathsf{p}(h)}^{h}  \ar@<-0.3ex>@{^{(}->}[rr]^-{\times x_{\mathsf{p}(h)-1}} &&
\mathcal{A}_{\mathsf{p}(h)-1}^{h} \ar@<-0.3ex>@{^{(}->}[rr]^-{\times x_{\mathsf{p}(h)-2}} && 
 \cdot \cdot \cdot \ar@<-0.3ex>@{^{(}->}[r]^-{\times x_2} & \mathcal{A}_2^{h} \ar@<-0.3ex>@{^{(}->}[r]^-{\times x_1}  & \mathcal{A}_1^{h} = \mathcal{R}/I_h 
 }
\end{equation}
of the ring $\mathcal{R}/I_h \cong H^*(\Hess(\mathsf{N},h))$. Moreover, for each $s$ with $1 \leq s \leq \mathsf{p}(h)$ we have 
\[
\mathcal{A}^{h}_s / \mathcal{A}^{h}_{s+1} \cong \mathcal{A}_{r_s}^{h^{(s)}}
\]
where by slight abuse of notation we denote by $\mathcal{A}_{s+1}^{h}$ the image of $\mathcal{A}_{s+1}^{h}$ in $\mathcal{A}_s^{h}$ under the map $\times x_s$. 
In particular, if $s=\mathsf{p}(h)$, then we have the isomorphism 
\[
\mathcal{A}^h_{\mathsf{p}(h)} \cong \mathcal{A}_{r_{\mathsf{p}(h)}}^{h^{(\mathsf{p}(h))}}.
\]

\end{theorem} 

The second statement in the theorem above says that the graded pieces of the associated graded ring corresponding to the filtration~\eqref{eq: filtration main} are isomorphic to analogous rings that are associated to a Hessenberg function for the integer $n-1$. In this sense, the graded pieces correspond to ``smaller'' Hessenberg functions, allowing us to make inductive arguments. This will be important for both the proof of the above theorem as well as for the corollaries thereof.

%%%%%%%%%%%%%%%%%%%%%%%%%%%%%%%%%%
\section{Proof of main theorem, part 1: a filtration of $R/I_h$}
\label{section:exact sequence}
%%%%%%%%%%%%%%%%%%%%%%%%%%%%%%%%%%

In this section, we prove the first statement of Theorem~\ref{theorem: main}, namely, that there exists a filtration of the ring $R/I_h$ as in~\eqref{eq: filtration main}. 
We present the argument as a sequence of lemmas, from which the filtration~\eqref{eq: filtration main} follows straightforwardly.

\begin{lemma} \label{lemma:4-1}
Let $1 \leq j \leq i \leq n$. Then 
$$
f_{i,j} \equiv x_j \cdot g_{i,j} \ \ \ {\rm mod} \ g_{i,j-1}
$$
in $\mathcal{R} = \Q[x_1,\ldots,x_n]$. Equivalently, $f_{i,j} - x_j \cdot g_{i,j} \in \langle g_{i,j-1}\rangle$. 
\end{lemma}

\begin{proof}
We take cases. First suppose $j=1$. Then by our conventions $g_{i, j-1} = g_{i,1-1} = g_{i,0}=0$. 
Therefore, in this case we must prove that $f_{i,1} = g_{i,1} \cdot x_1$ (i.e. that the LHS and RHS are actually equal, not just equivalent modulo some equivalence relation). 
By definition 
\[
f_{i,1} := \left( \prod_{\ell=2}^i (x_1 - x_\ell) \right) x_1 \quad \textup{ and } \quad 
g_{i,1} := \prod_{\ell=2}^i (x_1 - x_\ell) 
\]
where the products are viewed to be equal to $1$ in the case that $i=1<2$. It follows immediately that for any $1 \leq i \leq n$ we have $f_{i,1} = g_{i,1} \cdot x_1$ as desired. 

Next we consider the case $j \geq 2$. In this case we have by definition 
\begin{equation}\label{eq: gij-1} 
\begin{split}
g_{i,j-1}&=\sum_{k=1}^{j-1}  \prod_{\ell=j}^i (x_k-x_\ell) \\ 
&=\sum_{k=1}^{j-1} \left( \prod_{\ell=j+1}^i (x_k-x_\ell) \right) x_k - \sum_{k=1}^{j-1} \left( \prod_{\ell=j+1}^i (x_k-x_\ell) \right) x_j
\end{split}
\end{equation}
where the second equality is obtained by splitting off the factor $(x_k-x_j)$ in the product formula, and products are interpreted to be equal to $1$ if the top index is smaller than the bottom index. 
We then have 
\begin{align*}
f_{i,j}&=\sum_{k=1}^{j} \left( \prod_{\ell=j+1}^i (x_k-x_\ell) \right)x_k \\ 
&=\sum_{k=1}^{j-1} \left( \prod_{\ell=j+1}^i (x_k-x_\ell) \right)x_k + \left( \prod_{\ell=j+1}^i (x_j-x_\ell) \right)x_j \\ 
&\equiv \sum_{k=1}^{j-1} \left( \prod_{\ell=j+1}^i (x_k-x_\ell) \right) x_j + \left( \prod_{\ell=j+1}^i (x_j-x_\ell) \right)x_j \ \ \ {\rm modulo} \ g_{i,j-1}  \textup{ by~\eqref{eq: gij-1} } \\ 
& = \sum_{k=1}^j \left( \prod_{\ell=j+1}^i (x_k - x_\ell) \right) x_j \\
&=x_j \cdot g_{i,j} \quad \textup{ by definition of $g_{i,j}$ } 
\end{align*}
as desired. 
\end{proof}

The following lemma is based on the proof of \cite[Lemma~4.1]{AHHM}.

\begin{lemma} \label{lemma:4-2}
Let $h: [n] \to [n]$ be a Hessenberg function. 
For $1 < j \leq n$ and $h(j-1) \leq i \leq n$ we have 
\begin{align*}
f_{i,j-1} \in \langle f_{h(1),1},\ldots,f_{h(j-1),j-1} \rangle, \\ 
g_{i,j-1} \in \langle g_{h(1),1},\ldots,g_{h(j-1),j-1} \rangle. 
\end{align*}
\end{lemma}

\begin{proof}
We induct on $j$ and $i$. Note that $j$ is assumed to satisfy $1 < j$ so the base case is $j=2$. 
In this case we wish to show that 
\[
f_{i,1} \in \langle f_{h(1),1} \rangle
\]
for $h(1) \leq i \leq n$. If $i=h(1)$, then $f_{i,1}=f_{h(1),1}$ so the claim is immediate. Otherwise, the claim follows by a simple induction on $i$ since we know from Lemma~\ref{lemma: recursion} that 
$f_{i,1} = (x_1-x_i) f_{i-1,1}$. Thus $f_{i,1} \in \langle f_{i-1,1} \rangle \subseteq \langle f_{i-2,1} \rangle \subseteq \cdots \subseteq \langle f_{h(1),1} \rangle$ as desired. The argument for $g_{i,1} \in \langle g_{h(1),1}\rangle$ is similar. 

Now suppose that $j \geq 3$ and assume by induction that the claim is true for $j-1$, i.e., for any $i$ with $h(j-2) \leq i \leq n$ we know 
$f_{i,j-2} \in \langle f_{h(1),1}, \ldots, f_{h(j-2), j-2} \rangle$. 
We then wish to show the claim for $j$ and any $i$ with $h(j-1) \leq i \leq n$. We induct on $i$.
The base case is $i=h(j-1)$, which is clear. Now assume $h(j-1)<i \leq n$ and assume
the claim is known for $i-1$. By Lemma~\ref{lemma: recursion} we have 
\begin{equation}\label{eq: ideal recursion} 
f_{i,j-1} = f_{i-1, j-2} + (x_{j-1} - x_i) f_{i-1,j-1}.
\end{equation}
By the inductive hypothesis on the $i$, we know that $f_{i-1,j-1} \in \langle  f_{h(1),1}, \ldots, f_{h(j-1),j-1} \rangle$. 
Next observe that since $i>h(j-1)$ we know $i-1 \geq h(j-1) \geq h(j-2)$. Hence by the inductive hypothesis on $j$ we know that 
\[
f_{i-1, j-2} \in \langle f_{h(1),1}, \ldots, f_{h(j-2), j-2} \rangle \subseteq \langle f_{h(1),1}, \ldots, f_{h(j-1), j-1}\rangle.
\]
Hence we conclude from~\eqref{eq: ideal recursion} that 
\[
f_{i,j-1} \in \langle f_{h(1),1}, \ldots, f_{h(j-1), j-1}\rangle
\]
 as desired. By Lemma~\ref{lemma: recursion} the recursive formula for the $g$'s  are the same as for the $f$'s, so the proof for the $g_{i,j-1}$ is similar. 
 \end{proof}

Using the above two lemmas, we get the following.  

\begin{lemma}\label{lemma: equal ideals} 
Let $1 \leq j \leq n$. The ideal 
\begin{equation*}
\langle g_{h(1),1},\ldots,g_{h(j-1),j-1},f_{h(j),j},f_{h(j+1),j+1},\ldots,f_{h(n),n}\rangle
\end{equation*}
is equal to the ideal 
\begin{equation*}
\langle g_{h(1),1},\ldots,g_{h(j-1),j-1},x_{j} \cdot g_{h(j),j},f_{h(j+1),j+1},\ldots,f_{h(n),n} \rangle.
\end{equation*}
\end{lemma} 

\begin{proof} 
In the case $j=1$ it suffices to observe that the definitions of $f_{i,j}$ and $g_{i,j}$ immediately imply that $f_{h(1),1} = x_1 \cdot g_{h(1),1}$. This in turn implies that the two ideals are equal. 
For $j>1$ we have  
\begin{equation*}
\begin{split} 
&\langle g_{h(1),1},\ldots,g_{h(j-1),j-1},f_{h(j),j},f_{h(j+1),j+1},\ldots,f_{h(n),n}\rangle \\
=&\langle g_{h(1),1},\ldots,g_{h(j-1),j-1},g_{h(j),j-1},f_{h(j),j},f_{h(j+1),j+1},\ldots,f_{h(n),n} \rangle  \quad 
\textup{ by Lemma~\ref{lemma:4-2}} \\
 & \phantom{move over move over } \textup{ since $h(j) \geq h(j-1)$} \\
=&\langle g_{h(1),1},\ldots,g_{h(j-1),j-1},g_{h(j),j-1},x_{j} \cdot g_{h(j),j},f_{h(j+1),j+1},\ldots,f_{h(n),n} \rangle
\quad \textup{ by Lemma~\ref{lemma:4-1} }  \\
=&\langle g_{h(1),1},\ldots,g_{h(j-1),j-1},x_{j} \cdot g_{h(j),j},f_{h(j+1),j+1},\ldots,f_{h(n),n} \rangle \quad 
\textup{ by Lemma~\ref{lemma:4-2}} 
\end{split}
\end{equation*}
as desired. 
\end{proof}

We now recall the following fact. 
\begin{lemma}\label{lemma: f regular}(\cite[Lemma 6.8]{AHHM}) 
The polynomials $\{f_{h(1),1}, f_{h(2),2}, \ldots, f_{h(n),n}\}$ form a regular sequence in $\Q[x_1,\ldots,x_n]$. 
\end{lemma}

In our arguments below, we consider ideals generated by sequences of polynomials which replace some of the $f_{h(j),j}$ appearing in the sequence above with $g_{h(j),j}$'s, or $x_j \cdot g_{h(j),j}$. The following lemma states that these still form regular sequences. 

\begin{lemma}\label{lemma:4-4}
Let $h: [n] \to [n]$ be a Hessenberg function. 
Then
\begin{enumerate} 
\item[(i)] The set $\{g_{h(1),1},g_{h(2),2},\dots,g_{h(j-1),j-1},x_j \cdot
g_{h(j),j},f_{h(j+1),j+1},\dots,f_{h(n),n}\}$ is a regular sequence
 for any integer $1 \leq j \leq p(h)$.
\item[(ii)] The set
$\{g_{h(1),1},g_{h(2),2},\dots,g_{h(j-1),j-1},g_{h(j),j},f_{h(j+1),j+1},\dots,f_{h(n),n}\}$
is a regular sequence.
for any integer $0
\leq j <p(h)$. (When $j=0$ we interpret the sequence to be $\{f_{h(1),1},\ldots, f_{h(n),n}\}$.)

\end{enumerate} 
\end{lemma}

\begin{proof} 
By Proposition~\ref{prop: regular seq in terms of Hilbert}, if $\{F_1,\dots,F_n\}$ and 
$\{G_1,\ldots, G_n\}$ are two sequences of polynomials with $\deg(F_i)=\deg(G_i)$ for all $i$ 
and 
\[
\langle F_1,\dots,F_n \rangle =\langle
G_1,\dots,G_n \rangle
\]
and if in addition $\{F_1,\ldots,F_n\}$ is a regular sequence, then 
$G_1,\dots,G_n$ is also a regular sequence.
Thus by Lemma~\ref{lemma: equal ideals} if (ii) holds for $j=k$ then (i) holds for $j=k+1$.
Also, from Lemma~\ref{lemma: peel} we know that if (i) holds for $j=k$ then (ii) also holds
for $j=k$. Since the statement (ii) for $j=0$ is Lemma~\ref{lemma: f regular}, 
an induction argument starting at $j=0$ for statement (ii) implies the claims. 
\end{proof}

Using the above lemmas, we can now 
prove the first part of Theorem~\ref{theorem: main}, namely, that there is an exact sequence connecting the rings $\mathcal{A}_s^{h}$ for differing values of $s$. We state this formally as a Proposition below.

\begin{proposition}\label{prop: first half of main theorem}
Let $s$ be an integer, $1 \leq s \leq \mathsf{p}(h)$. 
Following the notation established above, there is a well-defined injective linear map
\[
\xymatrix{
\mathcal{A}^{h}_{s+1} \ar@<-0.3ex>@{^{(}->}[r]^-{\times x_s} & \mathcal{A}_s^{h} 
}\]
which fits into an exact sequence 
\begin{equation} \label{equation:exact}
0 \to \mathcal{A}_{s+1}^{h} \xrightarrow{\times x_{s}} \mathcal{A}_s^{h} \to \mathcal{A}_s^{h}/\langle x_{s} \rangle \to 0.
\end{equation}
Thus there exists a filtration 
\begin{equation}\label{eq: filtration prop}
\xymatrix {
0  \ar@<-0.3ex>@{^{(}->}[r] & \mathcal{A}_{\mathsf{p}(h)}^{h}  \ar@<-0.3ex>@{^{(}->}[rr]^-{\times x_{\mathsf{p}(h)-1}} &&
\mathcal{A}_{\mathsf{p}(h)-1}^{h} \ar@<-0.3ex>@{^{(}->}[rr]^-{\times x_{\mathsf{p}(h)-2}} && 
 \cdot \cdot \cdot \ar@<-0.3ex>@{^{(}->}[r]^-{\times x_2} & \mathcal{A}_2^{h} \ar@<-0.3ex>@{^{(}->}[r]^-{\times x_1}  & \mathcal{A}_1^{h} = \mathcal{R}/I_h 
 }
\end{equation}
of the ring $\mathcal{R}/I_h \cong H^*(\Hess(\mathsf{N},h))$. 
\end{proposition}

\begin{proof} 
Recall that by definition 
\[
\mathcal{A}_{s+1}^{h} := \mathcal{R}/\langle g_{h(1),1},\ldots,g_{h(s),s},f_{h(s+1),s+1},\ldots,f_{h(n),n} \rangle 
\]
and 
\[
\mathcal{A}_s^{h} := \mathcal{R}/\langle g_{h(1),1},\ldots,g_{h(s-1),s-1},f_{h(s),s},\ldots,f_{h(n),n} \rangle.
\]
We know from Lemma~\ref{lemma: equal ideals} that  
\[
\mathcal{A}_s^{h} = \mathcal{R}/\langle g_{h(1),1}, \cdots, g_{h(s-1),s-1}, x_s \cdot g_{h(s),s}, f_{h(s+1),s+1},\cdots, f_{h(n),n} \rangle.
\]
Moreover, from Lemma~\ref{lemma:4-4} we know that $\{g_{h(1),1}, g_{h(2),2}, \ldots, g_{h(s-1),s-1}, x_s \cdot g_{h(s),s}, f_{h(s+1),s+1}, \ldots, f_{h(n),n}\}$ is a regular sequence.  
Now by Lemma~\ref{lemma:key} we conclude that 
the map 
\[
\mathcal{A}_{s+1}^{h} 
\xrightarrow{\times x_{s}} 
\mathcal{R}/\langle g_{h(1),1}, \ldots, g_{h(s-1),s-1}, x_s \cdot g_{h(s),s}, f_{h(s+1),s+1},\ldots,f_{h(n),n} \rangle
\]
which is defined by multiplication by $x_s$ 
is well-defined and injective. Thus the first statement of the proposition is proved. 
The exactness of~\eqref{equation:exact} is given by the exactness of~\eqref{eq: key exact seq}. 
The filtration~\eqref{eq: filtration prop} follows immediately by considering the injections in sequence. 
\end{proof}

%%%%%%%%%%%%%%%%%%%%%%%%%%%%%%%%%%
\section{Proof of main theorem, part 2: the quotient ring}
\label{section:quotient ring}
%%%%%%%%%%%%%%%%%%%%%%%%%%%%%%%%%%

Proposition~\ref{prop: first half of main theorem} of the previous section proves the first part of Theorem~\ref{theorem: main}. In this section, we prove the second half. 
Specifically, we give a description of the quotient ring $\mathcal{A}_s^{h}/\langle x_s \rangle$ as an analogous ring that is associated to a ``smaller'' Hessenberg function from $[n-1]$ to $[n-1]$. Before continuing we note that from the exact sequence~\eqref{equation:exact} we know that there exists an isomorphism 
\[
\mathcal{A}_s^{h}/\mathcal{A}_{s+1}^{h} \cong \mathcal{A}_s^{h}/\langle x_s \rangle
\]
where by slight abuse of notation 
we denote by $\mathcal{A}_{s+1}^{h}$ the image of $\mathcal{A}_{s+1}^{h}$ in $\mathcal{A}_s^{h}$ under the map $\times x_s$.  We will therefore use $\mathcal{A}_s^{h}/\mathcal{A}_{s+1}^{h}$ and $\mathcal{A}_s^{h}/\langle x_s \rangle$ interchangeably in the discussion that follows. 
We have the following. 

\begin{proposition} \label{proposition:5-1}
Let $h$ be a Hessenberg function and $1 \leq s \leq \mathsf{p}(h)$. Then the following ring isomorphism holds
\begin{equation}\label{eq: quotient ring iso}
\mathcal{A}_s^{h}/\mathcal{A}_{s+1}^{h} \cong \mathcal{A}_{r_{s}}^{h^{(s)}}
\end{equation}
where $r_s$ is defined in~\eqref{eq: def rs}. 
Here, the ring on the RHS is  
$$
\Q[y_1,\ldots,y_{n-1}]/\langle g_{h^{(s)}(1),1}(y),\ldots,g_{h^{(s)}(r_{s}-1),r_{s}-1}(y),f_{h^{(s)}(r_{s}),r_{s}}(y),\ldots,f_{h^{(s)}(n-1),n-1}(y) \rangle 
$$
and the isomorphism of~\eqref{eq: quotient ring iso} is realized by sending $x_m$ to $y_m$ for $1 \leq m \leq s-1$ and $x_m$ to $y_{m-1}$ for $s+1 \leq m \leq n$.
 In particular, if $s=\mathsf{p}(h)$, then we have
$$
\mathcal{A}_{\mathsf{p}(h)}^{h} \cong \mathcal{A}_{r_{\mathsf{p}(h)}}^{h^{(\mathsf{p}(h))}}.
$$
\end{proposition}

Before proving Proposition~\ref{proposition:5-1}, we prove the following lemma, which is a simple consequence of Lemmas~\ref{lemma:4-1} and~\ref{lemma:4-2}.

\begin{lemma} \label{lemma:4-3}
Let $h: [n] \to [n]$ be a Hessenberg function.
Then for $1 \leq m \leq n$ we have 
\begin{align*}
f_{h(m),m} \in \langle g_{h(1),1},\ldots,g_{h(m),m} \rangle.
\end{align*}
In particular, we have
$$
\langle f_{h(1),1}, f_{h(2),2}, \ldots, f_{h(m),m} \rangle \subset \langle g_{h(1),1}, g_{h(2),2}, \ldots, g_{h(m),m} \rangle
$$
for $1 \leq m \leq n$.
\end{lemma}

\begin{proof}
From Lemma~\ref{lemma:4-1} we conclude that 
\[
f_{h(m),m} \in \langle g_{h(m),m},g_{h(m),m-1} \rangle. 
\]
Furthermore, since $h(m) \geq h(m-1)$ by the definition of Hessenberg functions, from Lemma~\ref{lemma:4-2} we can conclude that 
\[
g_{h(m), m-1} \in \langle g_{h(1),1}, \ldots, g_{h(m-1), m-1} \rangle.
\]
Putting the above together we obtain that 
\[
f_{h(m),m} \in \langle g_{h(1),1}, \ldots, g_{h(m-1),m-1}, g_{h(m),m} \rangle
\]
as was to be shown. The second statement follows immediately. 
\end{proof}

We can now prove Proposition~\ref{proposition:5-1}.

\begin{proof}[Proof of Proposition~\ref{proposition:5-1}]

By the exact sequence~\eqref{equation:exact}, we know that 
\[
\mathcal{A}_s^{h}/\mathcal{A}_{s+1}^{h} \cong \mathcal{A}_s^{h}/\langle x_s \rangle.
\]
Note that in the special case $s=\mathsf{p}(h)$, the ring $\mathcal{A}_{\mathsf{p}(h)+1}^h$ is the zero ring, so the above statement is equivalent to $\mathcal{A}_{\mathsf{p}(h)}^h \cong \mathcal{A}_{\mathsf{p}(h)}^h/\langle x_{\mathsf{p}(h)} \rangle$, and the claim of the proposition is that $\mathcal{A}_{\mathsf{p}(h)}^h \cong  \mathcal{A}_{r_{\mathsf{p}(h)}}^{h^{(\mathsf{p}(h))}}$. 
Except for this difference in the statement of the claim, the argument we give below works for both the special case $s=\mathsf{p}(h)$ and for the other cases when $s<\mathsf{p}(h)$. 

To prove~\eqref{eq: quotient ring iso}, we define an explicit map $\varphi: \mathcal{A}_s^{h}/\langle x_s \rangle \to 
\mathcal{A}^{h^{(s)}}_{r_s}$. As stated in the statement of the theorem, we think of $\mathcal{A}^{h^{(s)}}_{r_s}$ as being realized as a quotient of the ring $\Q[y_1, y_2,\ldots, y_{n-1}]$ by the ideal 
\begin{equation}\label{eq: target ideal} 
\langle g_{h^{(s)}(1),1}(y),\ldots,g_{h^{(s)}(r_{s}-1),r_{s}-1}(y),f_{h^{(s)}(r_{s}),r_{s}}(y),\ldots,f_{h^{(s)}(n-1),n-1}(y) \rangle.
\end{equation}
We define $\varphi(x_m) = y_m$ if $1 \leq m \leq s-1$ and we define $\varphi(x_m) = y_{m-1}$ if $s+1 \leq m \leq n$. (Here by slight abuse of notation we denote by $x_j$ and $y_j$ their corresponding images in the appropriate quotient rings.) We define $\varphi(x_s)=0$. 
We need to prove that $\varphi$ is a well-defined ring map, and that it is an isomorphism. We first check well-definedness. To do this, it suffices to check that $\varphi$ takes the defining relations of $\mathcal{A}_s^{h}/\langle x_s \rangle$ to the ideal~\eqref{eq: target ideal}. 
 In this context we may view the defining relations of $\mathcal{A}_s^{h}/\langle x_s \rangle$ to be the polynomials \[
 g_{h(1),1}, g_{h(2),2}, \cdots, g_{h(s-1),s-1}, f_{h(s),s}, \cdots, f_{h(n),n}\]
  considered as polynomials in $x_1,\ldots,x_n$, where we additionally set $x_s=0$.

We take cases. In each case, we must consider what happens to the $m$-th generator from the list above.

\smallskip
\noindent \textbf{Case (i):}  Suppose that $1 \leq m \leq r_{s}-1$. 
Recall that $r_s = \mathrm{min} \{m \in [n] \, \mid \, h(m) \geq s\}$ and from Lemma~\ref{lemma: rs} we have $r_s \leq s$.
Hence $1 \leq m \leq r_s - 1 \leq s-1$. 
Moreover, by definition of $r_s$, if $m \leq r_s-1$ then $h(m) < s$. 
By the definition of $\varphi$ we conclude that $\varphi(x_k)=y_k$ for all $1 \leq k \leq h(m)$. Moreover, by the definition of $h^{(s)}$ we have $h^{(s)}(m) = h(m)$ in this case. Thus we have 
\begin{align*}
\varphi(g_{h(m),m}(x)) & = \varphi \left(\sum_{k=1}^m \left( \prod_{\ell=m+1}^{h(m)} (x_k-x_\ell) \right) \right) \\
& =\sum_{k=1}^m \left( \prod_{\ell=m+1}^{h^{(s)}(m)} (y_k-y_\ell) \right) =g_{h^{(s)}(m),m}(y). 
\end{align*}
Thus $\varphi$ takes $g_{h(m),m}(x)$ to an element in~\eqref{eq: target ideal} as desired. 

\smallskip

\noindent \textbf{Case (ii):} Suppose that $r_{s} \leq m \leq s-1$. Then $h(m) \geq s$.
As in the above case we may compute 
\begin{align*}
\varphi(g_{h(m),m}(x))&=\varphi\left( \sum_{k=1}^m \left( \prod_{\ell=m+1}^{h(m)} (x_k-x_\ell) \right) \right) \\
&\equiv \sum_{k=1}^m \left[ \left( \prod_{\ell=m+1}^{s-1} (\varphi(x_k)-\varphi(x_\ell)) \right) \varphi(x_k) \left( \prod_{\ell=s+1}^{h(m)} (\varphi(x_k)-\varphi(x_\ell)) \right) \right] \\
& \phantom{move over}  \ \ \ \textup{ since $\varphi(x_s)=0$ } \\
&=\sum_{k=1}^m \left[ \left( \prod_{\ell=m+1}^{s-1} (y_k-y_\ell) \right) y_k \left( \prod_{\ell=s+1}^{h(m)} (y_k-y_{\ell-1}) \right) \right]  \ \ \  \textup{ by definition of $\varphi$} \\
&=\sum_{k=1}^m \left( \prod_{\ell=m+1}^{h(m)-1} (y_k-y_\ell) \right) y_k \\
& = f_{h(m)-1,m}(y) \\ 
&=f_{h^{(s)}(m),m}(y) \ \ \ \textup{ by definition of $h^{(s)}$ }
\end{align*}
since $r_s \leq m \leq s-1$. 
Since $m \geq r_s$ by assumption, we know that $f_{h^{(s)}(m),m}(y)$ is an element of~\eqref{eq: target ideal}, as desired. 

\smallskip

\noindent \textbf{Case (iii):}  Suppose that $m=s$. Then we have 
\begin{align*} 
\varphi(f_{h(s),s}(x)) & = \varphi\left( \sum_{k=1}^s \left(\prod_{\ell=s+1}^{h(s)} (x_k - x_\ell) \right) x_k \right) \\
& = \varphi\left( \sum_{k=1}^{s-1} \left(\prod_{\ell=s+1}^{h(s)} (x_k - x_\ell) \right) x_k \right) \ \ \ \textup{ since 
$\varphi(x_s) = 0$ } \\
& = \sum_{k=1}^{s-1} \left( \prod_{\ell=s+1}^{h(s)} (\varphi(x_k) - \varphi(x_\ell)) \right) \varphi(x_k) \\
& = \sum_{k=1}^{s-1} \left( \prod_{\ell=s+1}^{h(s)} (y_k - y_{\ell-1}) \right) y_k \ \ \ \textup{ by definition of $\varphi$} \\
& = \sum_{k=1}^{s-1} \left( \prod_{\ell=s}^{h(s)-1} (y_k - y_{\ell}) \right) y_k \\
& = f_{h(s)-1, s-1}(y). \\ 
\end{align*} 
Now note that $h^{(s)}(s-1) = h(s-1)-1$ by definition of $h^{(s)}$ so $h(s)-1 \geq h(s-1)-1 = h^{(s)}(s-1)$. Therefore by Lemma~\ref{lemma:4-2} we may conclude $\varphi(f_{h(s),s}(x)) = f_{h(s)-1,s-1}(y) \in \langle f_{h^{(s)}(1),1}, \cdots, f_{h^{(s)}(s-1), s-1} \rangle$. 
Applying Lemma~\ref{lemma:4-3} we can see that 
\[
\langle f_{h^{(s)}(1),1}, \ldots, f_{h^{(s)}(r_s-1),r_s-1} \rangle \subseteq \langle g_{h^{(s)}(1), 1}, \cdots, g_{h^{(s)}(r_s-1), r_s-1} \rangle 
\]
where all polynomials are in the $y$ coordinates. Therefore 
\[
\langle f_{h^{(s)}(1),1}, \ldots, f_{h^{(s)}(r_s-1),r_s-1} \rangle \subseteq \langle g_{h^{(s)}(1), 1}, \cdots, g_{h^{(s)}(r_s-1), r_s-1}, f_{h^{(s)}(r_s), r_s}, \cdots, f_{h^{(s)}(n-1),n-1} \rangle 
\]
where the RHS is the defining ideal of $\mathcal{A}_{r_s}^{h^{(s)}}$, so $\varphi(f_{h(s),s})$ also lies in this ideal, as desired. 

\smallskip

\noindent \textbf{Case (iv):} Suppose that $s+1 \leq m \leq n$. (Note that in the special case when $s=\mathsf{p}(h)$, this case is vacuous.) Then 
\begin{align*}
\varphi(f_{h(m),m}(x))&= \varphi\left( \sum_{k=1}^m \left( \prod_{\ell=m+1}^{h(m)} (x_k-x_\ell) \right)x_k \right) \\
&= \varphi\left( \sum_{k=1}^{s-1} \left( \prod_{\ell=m+1}^{h(m)} (x_k-x_\ell) \right)x_k \right) + 
\varphi \left( \sum_{k=s+1}^m \left( \prod_{\ell=m+1}^{h(m)} (x_k-x_\ell) \right)x_k \right) \\
 & \phantom{move over} \ \ \ \textup{ since $\varphi(x_s) = 0$ } \\
&= \sum_{k=1}^{s-1} \left( \prod_{\ell=m+1}^{h(m)} (y_k-y_{\ell-1}) \right)y_k+\sum_{k=s+1}^m \left( \prod_{\ell=m+1}^{h(m)} (y_{k-1}-y_{\ell-1}) \right)y_{k-1} \\
 & \phantom{move over} \ \ \ \textup{ by definition of $\varphi$} \\
&= \sum_{k=1}^{m-1} \left( \prod_{\ell=m+1}^{h(m)} (y_k-y_{\ell-1}) \right)y_k \\
&= \sum_{k=1}^{m-1} \left( \prod_{\ell=m}^{h(m)-1} (y_k-y_{\ell}) \right)y_k \\
&= f_{h(m)-1, m-1}(y) \\ 
& = f_{h^{(s)}(m-1), m-1}(y) \ \ \ \textup{ by definition of $h^{(s)}$ since $s\leq m-1 \leq n-1$.} \\
\end{align*}
Hence $\varphi(f_{h(m),m}(x))$ is also contained in~\eqref{eq: target ideal}.

 \smallskip
 
This completes the proof that $\varphi$ is well-defined.  In fact, in cases (i), (ii) and (iv) above 
we have proven something stronger. Namely, we have shown that 
\begin{equation}\label{eq: correspondence} 
\begin{split} 
\varphi(g_{h(m),m}(x)) & = \begin{cases}
g_{h^{(s)}(m),m}(y) \ \ \ {\rm if} \ 1 \leq m \leq r_s-1, \\
f_{h^{(s)}(m),m}(y) \ \ \ {\rm if} \ r_s \leq m \leq s-1, \\
\end{cases} \\
\varphi( f_{h(m),m}(x) ) & = f_{h^{(s)}(m-1),m-1}(y) \ \ \ {\rm if} \ s+1 \leq m \leq n. 
\end{split} 
\end{equation}

To see that $\varphi$ is an isomorphism, note first that~\eqref{eq: correspondence} sets up an exact correspondence between generators of~\eqref{eq: target ideal} and generators of the ideal for $\mathcal{A}^{h}_s$. In particular, by taking preimages, it follows straightforwardly that any element lying in the kernel of $\varphi$ (viewed as a map from $\Q[x_1,\ldots,x_n]$ to $\mathcal{A}^{h^{(s)}}_{r_s}$) already lies in the ideal of relations for $\mathcal{A}^{h}_s$. This shows that $\varphi: \mathcal{A}^{h}_s/\langle x_s\rangle \to \mathcal{A}^{h^{(s)}}_{r_s}$ is injective. On the other hand, $\varphi$ is clearly surjective, since its generators $y_1,\ldots, y_{n-1}$ all lie in the image of $\varphi$ by construction. 
This shows that $\varphi$ is an isomorphism, as desired. 
\end{proof}

This completes the proof of Theorem~\ref{theorem: main}.

%%%%%%%%%%%%%%%%%%%%%%%%%%%%%%%%%
\section{An inductive formula for Poincar\'e polynomials} 
\label{section: inductive poincare}
%%%%%%%%%%%%%%%%%%%%%%%%%%%%%%%%%%

In the previous two sections, we proved the main result of this manuscript, namely Theorem~\ref{theorem: main}. 
In the remainder of 
this paper, we derive several consequences of this result.  For what follows, it may be useful to visualize
our main result as a commutative diagram as follows: 

\begin{equation}\label{eq: filtration and quotients} 
\xymatrix @C=1.8pc{ 
0 \ar[r] & \mathcal{A}_{\mathsf{p}(h)}^{h} \ar[rr]^{\times x_{\mathsf{p}(h)-1}} \ar[d]^{\cong} && \mathcal{A}_{\mathsf{p}(h)-1}^{h} \ar[r]^{\times x_{\mathsf{p}(h)-2}} \ar[d] &  \cdots 
& \ar[r]^{\times x_2} & \mathcal{A}_2^{h} \ar[r]^{\times x_1} \ar[d] & \mathcal{A}^{h}_1 = \mathcal{R}/I_h \ar[d] \\
 & \mathcal{A}^{h^{(\mathsf{p}(h))}}_{r_{\mathsf{p}(h)}} && \mathcal{A}_{r_{\mathsf{p}(h)-1}}^{h^{(\mathsf{p}(h)-1)}} & & & \mathcal{A}_{r_2}^{h^{(2)}} &   \mathcal{A}_{r_1}^{h^{(1)}} 
}
\end{equation} 
 where the horizontal arrows are inclusions and the vertical arrows are surjections.

From~\eqref{eq: filtration and quotients} it is easily seen that we can express the Hilbert series of $\mathcal{R}/I_h \cong H^*(\Hess(\mathsf{N},h);\Q)$ inductively in terms of the Hilbert series of the $\mathcal{A}_{r_s}^{h^{(s)}}$ for varying $s$. 
Specifically, from Proposition~\ref{prop: regular seq in terms of Hilbert} and Lemma~\ref{lemma:4-4} we know that the polynomial 
\[
F_s^{h}(t) := \left( \prod_{m=1}^{s-1} (1+t+t^2 + \cdots + t^{h(m)-m-1}) \right) 
\left( \prod_{m=s}^n (1+t+t^2 + \cdots + t^{h(m)-m}) \right)
\]
is the Hilbert series $F(\mathcal{A}^{h}_s, t)$ of the ring $\mathcal{A}^{h}_s$. In particular, when $s=1$, we have 
\[
F_1^{h}(t) = F(\mathcal{A}^{h}_1, t) = \mathrm{Poin}(\Hess(\mathsf{N},h), \sqrt{t})
\]
where the RHS denotes the Poincar\'e polynomial of the variety $\Hess(\mathsf{N},h)$ in the variable $t$. (The square root $\sqrt{t}$ is due to the fact that $\Hess(\mathsf{N},h)$ has no odd-degree cohomology.) 
From the diagram~\eqref{eq: filtration and quotients} and noting that the multiplication maps shifts the degrees, we immediately obtain the following corollary of Theorem~\ref{theorem: main}. 

\begin{corollary}\label{corollary: inductive Poincare}
Let $h: [n] \to [n]$ be a Hessenberg function. Then 
\begin{equation}\label{eq: inductive poincare formula}
\mathrm{Poin}(\Hess(\mathsf{N},h),\sqrt{t}) =  t^{\mathsf{p}(h)-1} F_{r_{\mathsf{p}(h)}}^{h^{(\mathsf{p}(h))}}(t) +  t^{\mathsf{p}(h)-2} F_{r_{\mathsf{p}(h)-1}}^{h^{(\mathsf{p}(h)-1)}}(t) + \cdots + F_{r_1}^{h^{(1)}}(t).
\end{equation}
\end{corollary} 

\begin{proof} 
The Poincar\'e polynomial of $\Hess(\mathsf{N},h)$ is the same as the Poincar\'e polynomial of its associated graded ring corresponding to the filtration~\eqref{eq: filtration and quotients}, but each piece of the associated graded ring is isomorphic to a ring of the form $\mathcal{A}_{r_k}^{h^{(k)}}$, whose Poincar\'e polynomial is $F_{r_k}^{h^{(k)}}(t)$. The shift in degrees, reflected by the multiplication by $t^{k-1}$, is due to the shift in degrees that occurs in~\eqref{eq: filtration and quotients}, since each multiplication by $x_\ell$ increases the degree by $1$. 
\end{proof} 

\begin{remark}
Sommers and the fifth author also gave a similar formula 
\begin{equation} \label{eq:Sommers-Tymoczko}
\mathrm{Poin}(\Hess(\mathsf{N},h),\sqrt{t}) = \prod_{m=1}^n (1+t+\cdots + t^{h(m)-m})
\end{equation}
for the same Poincar\'e polynomial \cite{SommersTymoczko}.  In fact, it is possible to derive our Corollary~\ref{corollary: inductive Poincare} directly from their formula by some straightforward algebraic manipulations. 
 However, it was our filtration which made apparent (to us) the inductive nature of this Poincar\'e polynomial. 
\end{remark}

\begin{example}\label{example: poincare for 2344}
Let $h=(2,3,4,4)$. Then the diagram of the Hessenberg function is

\[\ytableausetup{centertableaux} 
h = (2,3,4,4):  \; \;   \begin{ytableau}
\star &  \star & \star & \star \\ 
\star & \star &\star & \star  \\
 & \star  & \star & \star \\ 
 &  & \star & \star  \\ 
\end{ytableau}
\]
and it is not hard to see that $h^{(1)} = (2,3,3), h^{(2)} = (1,3,3), h^{(3)} = (2,2,3)$ and $h^{(4)} = (2,3,3)$. It is also straightforward to compute that $r_1 = 1, r_2 = 1, r_3 = 2$ and $r_4 = 3$. We then obtain that 
\[
\begin{split} 
F_{r_1}^{h^{(1)}}(t)  = (1+t)^2, \ \ 
F_{r_2}^{h^{(2)}}(t)  = (1+t), \ \
F_{r_3}^{h^{(3)}}(t)  = 1, \ \
F_{r_4}^{h^{(4)}}(t)  = 1. 
\end{split} 
\]
Therefore Corollary~\ref{corollary: inductive Poincare} yields the formula 
\begin{align*}
\mathrm{Poin}(\Hess(\mathsf{N}, (2,3,4,4)),\sqrt{t}) &= t^3 \cdot 1 + t^2 \cdot 1 + t \cdot (1+t) + (1+t)^2 \\
&= (1+t)(1+t)(1+t)
\end{align*}
which is precisely the formula in \eqref{eq:Sommers-Tymoczko}, and where the last equality is obtained by a straightforward computation. 
\end{example}

%%%%%%%%%%%%%%%%%%%%%%%%%%%%%%%%%%
\section{A monomial basis for $H^*(\Hess(\mathsf{N},h))$}
\label{section:monomial basis}
%%%%%%%%%%%%%%%%%%%%%%%%%%%%%%%%%%

From the point of view of commutative algebra, it is a natural question to ask whether there exists a \emph{monomial} basis for any ring which is presented explicitly as a quotient ring $\mathcal{R}/I$ of a polynomial ring $\mathcal{R}=\Q[x_1,\ldots, x_n]$ modulo an ideal $I$. More precisely, the question is whether there exists a set of monomials $\{x^\alpha = x_1^{\alpha_1} x_2^{\alpha_2} \cdots x_n^{\alpha_n}\}_{\alpha \in \mathcal{S}}$ for some subset $\mathcal{S} \subseteq \Z_{\geq 0}^n$ such that the images of these monomials in $\mathcal{R}/I$ form an additive basis. Classical examples of such include the ``standard monomial bases" of Gr\"obner theory. 
Another example arises in the study of Springer fibers, which are Hessenberg varieties corresponding to the special case $h=(1,2,\ldots n)$.  These varieties carry an action of the symmetric group on their cohomology groups, and the top-dimensional cohomology group is an irreducible representation; moreover, each irreducible representation of the symmetric group can be obtained in this way.  A monomial basis for the cohomology of type A Springer fibers was defined by De Concini and Procesi in \cite{DeCP}.  Garsia and Procesi  then proved that this basis satisfies a number of remarkable properties, and used it to study the graded character for the Springer representation \cite{GarPro}.

In this section, we use our main theorem to construct a natural monomial basis for $H^*(\Hess(\mathsf{N},h)) \cong \mathcal{R}/I_h$.  In fact, we prove a stronger statement: we construct a monomial basis for any ring of the form $\mathcal{A}_s^{h}$ as introduced in Section~\ref{section:Preliminary}.

We note that (up to a change in conventions) this monomial basis was conjectured by Mbririka in \cite{Mbirika}.   In fact, Mbrika's conjecture was motivated by Garsia and Procesi's work on the cohomology of Springer fibers which was mentioned above, and we may view the basis below as an extension of their monomial basis to the setting of regular nilpotent Hessenberg varieties.  In the special case that $h=(n,n,\ldots,n)$, the monomials we obtain are the usual monomial basis for the cohomology of the flag variety appearing in the literature (e.g. \cite[Proposition 3 in Section 10.2]{Fulton}). We now have the following.

\begin{theorem} \label{theorem: monomial basis}
Let $n$ be a positive integer and let $h: [n] \to [n]$ be a Hessenberg function. Let $s$ be an integer, $1 \leq s \leq \mathsf{p}(h)$. 
Then the (image under the projection map $\mathcal{R} \to \mathcal{A}_s^{h}$ of the) following set of monomials 
$$
\left\{x_1^{i_1} \cdots x_{s-1}^{i_{s-1}} x_{s}^{i_{s}} \cdots x_n^{i_n} \, \left \lvert \,  \begin{array}{l}
0 \leq i_m \leq h(m)-m-1 \ \ \ {\rm if} \ 1 \leq m \leq s-1 \\
0 \leq i_m \leq h(m)-m \ \ \ {\rm if} \ s \leq m \leq n
\end{array} 
\right.\right\}
$$
is an additive basis for $\mathcal{A}_s^{h}$. 
\end{theorem}

To make the induction argument work, we need the following lemma. 

\begin{lemma}\label{lemma: induction step}
Let $h: [n] \to [n]$ be a Hessenberg function. Let $s$ be an integer with $1 \leq s \leq \mathsf{p}(h)$. 
Then $r_s \leq \mathsf{p}(h^{(s)})$. 
\end{lemma} 

\begin{proof} 
Since by definition $\mathsf{p}(h^{(s)}) = \min \{ m \, \mid \, h^{(s)}(m)=m\}$, in order to show the claim of the lemma, it suffices to show that for $1 \leq m < r_s$ we have $h^{(s)}(m) > m$. Note that $r_s \leq s$ by 
Lemma~\ref{lemma: rs},
 but by assumption $s \leq \mathsf{p}(h)$, so $r_s \leq \mathsf{p}(h)$. Hence if $m < r_s$ we know $m < \mathsf{p}(h)$, which in turn implies $h(m) > m$ (by definition of $\mathsf{p}(h)$). 
From the definition of $h^{(s)}$ in~\eqref{eq: def hs} we have that for $1 \leq m < r_s$ we have $h^{(s)}(m) = h(m)$. Thus we can conclude $h^{(s)}(m) = h(m) > m$ for $m < r_s$, as desired. 
\end{proof}

We can now prove the theorem.

\begin{proof}[Proof of Theorem~\ref{theorem: monomial basis}]
We prove Theorem~\ref{theorem: monomial basis} by induction on $n$ and decreasing induction on $s$. 

First consider the base case $n=1$. In this case, the only possible Hessenberg function is the identity $h(1)=1$ and $\mathcal{A}^{h}_1 = \Q[x]/\langle f_{1,1} =x \rangle \cong \Q$. A basis is given by $x^{h(1)-1} = x^0=1 \in \Q$. This proves the base case. 

We proceed to the inductive step. Suppose now that $n>1$ and that the claim holds for $n-1$, any Hessenberg function $h': [n-1] \to [n-1]$, and any $s'$ with $1 \leq s' \leq \mathsf{p}(h')$. 
We now use a decreasing induction argument on the index $s$. Consider the base case $s=\mathsf{p}(h)$. 
In this case we know that $\mathcal{A}_{\mathsf{p}(h)}^{h} \cong \mathcal{A}_{r_{\mathsf{p}(h)}}^{h^{(\mathsf{p}(h))}}$ from Proposition~\ref{proposition:5-1}. Moreover, $\mathcal{A}_{r_{\mathsf{p}(h)}}^{h^{(\mathsf{p}(h))}}$ is a ring associated to a Hessenberg function on $[n-1]$.  By Lemma~\ref{lemma: induction step} we may therefore apply the inductive hypothesis, and the following set of monomials 
\begin{equation}\label{eq: inductive basis for s=ph}
\left\{y_1^{i_1} \cdots y_{n-1}^{i_{n-1}} \left| \begin{array}{l}
0 \leq i_m \leq h^{(\mathsf{p}(h))}(m)-m-1 \ \ \ {\rm if} \ 1 \leq m \leq r_{\mathsf{p}(h)}-1 \\
0 \leq i_m \leq h^{(\mathsf{p}(h))}(m)-m \ \ \ {\rm if} \ r_{\mathsf{p}(h)} \leq m \leq n-1
\end{array} 
\right.\right\}
\end{equation} 
is an additive basis of $\mathcal{A}_{r_{\mathsf{p}(h)}}^{h^{(\mathsf{p}(h))}}$. 
Recall that the Hessenberg function $h^{(\mathsf{p}(h))}$ is defined as 
\begin{equation}\label{eq: def h ph m}
h^{(\mathsf{p}(h))}(m) := 
\begin{cases} 
h(m) \ \ \textup{ if } 1 \leq m \leq r_{\mathsf{p}(h)}-1 \\ 
h(m)-1 \ \ \textup{ if } r_{\mathsf{p}(h)} \leq m \leq \mathsf{p}(h)-1 \\ 
h(m+1)-1 \ \ \textup{ if } \mathsf{p}(h) \leq m \leq n-1.
\end{cases} 
\end{equation} 
From this it follows that, under the isomorphism which sends $x_m$ to $y_m$ for $1 \leq m \leq \mathsf{p}(h)-1$ and $x_m$ to $y_{m-1}$ for $\mathsf{p}(h)+1 \leq m \leq n$ (and $x_{\mathsf{p}(h)}$ goes to $0$), the monomials in~\eqref{eq: inductive basis for s=ph} may be identified with the following set of monomials: 
\[
\left\{ x_1^{i_1} \cdots x_{\mathsf{p}(h)-1}^{i_{\mathsf{p}(h)-1}} x_{\mathsf{p}(h)+1}^{i_{\mathsf{p}(h)+1}} \cdots x_n^{i_n} \, \left\lvert  \, 
\begin{array}{l} 
0 \leq i_m \leq h(m)-m-1 \ \ \textup{ if } 1 \leq m \leq \mathsf{p}(h)-1 \\
0\leq i_m \leq h(m)-m \ \ \textup{ if } \mathsf{p}(h)+1 \leq m \leq n
\end{array} 
\right.
\right\}. 
\]
Since $h(\mathsf{p}(h)) = \mathsf{p}(h)$ by definition of $\mathsf{p}(h)$, this is in turn equal to the set 
\[
\left\{ x_1^{i_1} \cdots x_{\mathsf{p}(h)-1}^{i_{\mathsf{p}(h)-1}} x_{\mathsf{p}(h)}^{i_{\mathsf{p}(h)}} x_{\mathsf{p}(h)+1}^{i_{\mathsf{p}(h)+1}} \cdots x_n^{i_n} \, \left\lvert \, 
\begin{array}{l} 
0 \leq i_m \leq h(m)-m-1 \ \ \textup{ if } 1 \leq m \leq \mathsf{p}(h)-1 \\
0\leq i_m \leq h(m)-m \ \ \textup{ if } \mathsf{p}(h) \leq m \leq n
\end{array} 
\right.
\right\}
\]
since the condition $0 \leq i_{\mathsf{p}(h)} \leq h(\mathsf{p}(h))-\mathsf{p}(h) = 0$ implies that $x_{\mathsf{p}(h)}$ never appears in these monomials. This is exactly the set given in the statement of the theorem, so we have proven the base case $s=\mathsf{p}(h)$.

 Now we assume that $s<\mathsf{p}(h)$ and that the claim holds for $s+1$.
By the inductive assumption, a basis of $\mathcal{A}_{s+1}^{h}$ is given by the monomials 
\begin{equation} \label{eq:6-1}
\left\{x_1^{i_1} \cdots x_n^{i_n} \left| \begin{array}{l}
0 \leq i_m \leq h(m)-m-1 \ \ \ {\rm if} \ 1 \leq m \leq s \\
0 \leq i_m \leq h(m)-m \ \ \ {\rm if} \ s+1 \leq m \leq n
\end{array} 
\right.\right\}.
\end{equation}
On the other hand, we have $\mathcal{A}_s^{h}/(x_{s}) \cong \mathcal{A}_{r_{s}}^{h^{(s)}}$ from Proposition~\ref{proposition:5-1}. Since $\mathcal{A}_{r_s}^{h^{(s)}}$ is a ring associated to a Hessenberg function on $[n-1]$ and because $r_s \leq \mathsf{p}(h^{(s)})$ by Lemma~\ref{lemma: induction step}, by the inductive assumption on $n$ we can take as a basis of $\mathcal{A}_s^{h}/(x_{s}) \cong \mathcal{A}_{r_{s}}^{h^{(s)}}$ the monomials 
\begin{equation*}
\left\{y_1^{i_1} \cdots y_{n-1}^{i_{n-1}} \left| \begin{array}{l}
0 \leq i_m \leq h^{(s)}(m)-m-1 \ \ \ {\rm if} \ 1 \leq m \leq r_s-1 \\
0 \leq i_m \leq h^{(s)}(m)-m \ \ \ {\rm if} \ r_s \leq m \leq n-1
\end{array} 
\right.\right\}. 
\end{equation*} 
Under the isomorphism $\varphi$ between $\mathcal{A}_s^{h}/(x_{s})$ and $\mathcal{A}_{r_{s}}^{h^{(s)}}$ constructed and used in the proof of Proposition~\ref{proposition:5-1}, the above monomials in the $y$ variables corresponds to the following set of monomials in the $x$ variables 
\begin{equation}\label{eq:6-2}
\left\{x_1^{i_1} \cdots x_{s-1}^{i_{s-1}}x_{s+1}^{i_{s+1}} \cdots x_n^{i_n} \left| \begin{array}{l}
0 \leq i_m \leq h(m)-m-1 \ \ \ {\rm if} \ 1 \leq m \leq s-1 \\
0 \leq i_m \leq h(m)-m \ \ \ {\rm if} \ s+1 \leq m \leq n
\end{array} 
\right.\right\}
\end{equation}
and from the isomorphism $\varphi$ it follows that~\eqref{eq:6-2} is an additive basis of $\mathcal{A}_s^{h}/(x_{s})$. Moreover, since the monomials above do not contain $x_s$, it follows that the set~\eqref{eq:6-2}, when viewed as elements in $\mathcal{A}_s^{h}$, are linearly independent. 
From the exact sequence in \eqref{equation:exact} it now follows that a basis of $\mathcal{A}_s^{h}$ can be obtained by combining the sets~\eqref{eq:6-2} and~\eqref{eq:6-1}, except that the set~\eqref{eq:6-1} must be multiplied by $x_s$. We conclude that 
$$
\big(\eqref{eq:6-1} \times x_s \big) \cup \eqref{eq:6-2}
$$
is an additive basis of $\mathcal{A}_s^h$. 
 It is straightforward to see that this set coincides with the set given in the statement of the theorem, since the set~\eqref{eq:6-2} gives precisely those monomials which do not contain an $x_s$, and the monomials obtained by multiplication by $x_s$ of the elements of~\eqref{eq:6-1} give precisely those monomials in which $x_s$ appears with an exponent between $1$ and $h(s)-s$. This completes the proof. 
\end{proof}

We state the special case when $s=1$, for which $\mathcal{A}_1^{h} = \mathcal{R}/I_h \cong H^*(\Hess(\mathsf{N},h))$, as a separate corollary.

\begin{corollary}\label{cor: basis}
Let $n$ be a positive integer and let $h: [n] \to [n]$ be a Hessenberg function. 
Then the (image under the projection map $\mathcal{R} \to H^*(\Hess(\mathsf{N},h)) \cong \mathcal{R}/I_h$ of the) following set of monomials 
\begin{equation}\label{eq: s=1 basis}
\{x_1^{i_1} \cdots x_n^{i_n} \mid 0 \leq i_m \leq h(m)-m \ {\rm for} \ 1 \leq m \leq n \}
\end{equation}
is an additive basis for $H^*(\Hess(\mathsf{N},h))$. 
\end{corollary}

\begin{remark} 
The monomial bases~\eqref{eq: s=1 basis} arising in Corollary~\ref{cor: basis} do not, in general, come from Gr\"obner theory, as can be seen by the following example. Let $n=3$ and $h=(2,3,3)$. In this case, the set of monomials~\eqref{eq: s=1 basis} is $\{1, x_1, x_2, x_1x_2\}$. Suppose for a contradiction that there exists a monomial order $<$ with respect to which the standard monomials corresponding to $\mathrm{init}_{<}(I_h)$ is this set. In order for this to occur it must be the case that $\mathrm{init}_{<}(f)$ for any $f \in I_h$ must be divisible by $x_1^2, x_2^2$, or $x_3$. (Here $\mathrm{init}_<(I)$ denotes the initial ideal of $I$ with respect to the monomial order $<$ as in standard Gr\"obner theory.) By definition of the generators of $I_h$ we know that both 
\[
(x_1-x_2)x_1=x_1^2-x_1x_2 
\]
and
\[
 (x_1-x_3)x_1+(x_2-x_3)x_2 +(x_1+x_2)(x_1+x_2+x_3)-2(x_1-x_2)x_1
=4x_1x_2+2x_2^2
\]
are elements of $I_h$. The initial term of the left element must be $x_1^2$ and the initial term of the right element must be $x_2^2$, which implies that under this monomial order we must have 
\[
x_1^2 > x_1x_2 \ \ \textup{ and } \ \  x_2^2 > x_1x_2. 
\]
By properties of monomial orders, this in turn implies that $x_1 > x_2$ and $x_2>x_1$, which is a contradiction. Therefore, there cannot exist any such monomial order $<$. 
\end{remark}

\begin{example}\label{example: basis for 2444}
Let $h=(2, 4,4,4)$. Then the corollary shows that the following set
\[
\{  1, x_1, x_2, x_3, x_1 x_2, x_1 x_3, x_2^2, x_2 x_3, x_1 x_2^2, x_1 x_2 x_3, x_2^2 x_3, x_1 x_2^2 x_3 \} 
\]
is an additive basis of $H^*(\Hess(\mathsf{N}, (2,4,4,4)) \cong \mathcal{R}/I_h$. 
\end{example}

%%%%%%%%%%%%%%%%%%%%
\section{Linear relations on Schubert classes in $H^*(\Hess(\mathsf{N},h))$} 
\label{sec: relations on Schuberts} 
%%%%%%%%%%%%%%%%%%%%%

We now give an algorithm for computing a basis of the set of linear relations on Schubert classes in the cohomology ring $H^*(\Hess(\mathsf{N},h))$. 
Recall from Theorem~\ref{theorem: cohomology} that there is a surjective ring homomorphism 
\[
H^*(\Flags(\C^n)) \to H^*(\Hess(\mathsf{N},h))
\]
induced from the inclusion map $\Hess(\mathsf{N},h) \into \Flags(\C^n)$. For the cohomology of the flag variety, there is a famous additive basis consisting of the Schubert classes $\{\sigma_w\}_{w \in S_n}$, parametrized by the permutations $w$ in the symmetric group $S_n$. Let $\overline{\sigma_w} \in H^*(\Hess(\mathsf{N},h))$ denote the image of the Schubert class $\sigma_w \in H^*(\Flags(\C^n))$ under the projection $H^*(\Flags(\C^n)) \to H^*(\Hess(\mathsf{N},h))$. Given that the projection is surjective, it is natural to ask whether there exists some natural subset of the Schubert classes which form an additive basis 
for $H^*(\Hess(\mathsf{N},h))$. 
The first and fifth authors were thinking about this problem some time ago and asked the following question. 

\begin{question} \label{conjecture:HaradaTymoczko}
Let $n$ be a positive integer and let $h: [n] \to [n]$ be a Hessenberg function. 
Does the following set of images of Schubert classes 
\begin{equation*}
\{\overline{\sigma_w} \mid w(m) \leq h(m) \ {\rm for} \ 1 \leq m \leq n \}
\end{equation*}
form an additive basis for $H^*(\Hess(\mathsf{N},h))$?
\end{question}

The results of \cite{HaradaTymoczko} give an answer to this question for the special case of the Peterson variety, which is the case when $h=(2,3,4,\cdots,n,n)$, i.e. $h(i)=i+1$ for $1 \leq i \leq n-1$ \cite[Theorem~4.12]{HaradaTymoczko}. However, as far as we are aware, the question is still open in the general case.

Motivated by this question, we can also ask a related question: what are the linear relations satisfied by the classes $\{\overline{\sigma_w}\}_{w \in \mathfrak{S}_n}$? In the remainder of this section, we address this question using techniques similar to those in the proof of 
Theorem~\ref{theorem: monomial basis}.

We need some terminology. 
Let $h, h'$ be two Hessenberg functions on $[n]$. We write $h' \subseteq h$ if $h'(i) \leq h(i)$ for all $1 \leq i \leq n$. In pictures, this is the situation when the ``star'' boxes in the diagram corresponding to $h'$ are also ``star'' boxes in the diagram corresponding to $h$. 

\begin{example} 
Let $h' = (2,3,3,5,5)$ and $h= (3,4,4,5,5)$. Then it is easy to check $h' \subseteq h$. We illustrate this in the diagram below where the ``stars'' correspond to the diagram of $h'$ and the shaded boxes indicate the diagram for $h$. We can then see that ``each star is contained in a shaded box''. 
\[\ytableausetup{centertableaux} 
  \; \;   \begin{ytableau}
*(grey) \star &  *(grey) \star & *(grey) \star & *(grey) \star & *(grey) \star \\ 
*(grey) \star & *(grey) \star & *(grey) \star & *(grey)  \star & *(grey) \star   \\
 *(grey) &  *(grey) \star  & *(grey) \star & *(grey) \star & *(grey) \star \\ 
 &  *(grey) & *(grey)  & *(grey) \star & *(grey) \star \\ 
 & &  & *(grey) \star & *(grey) \star \\ 
\end{ytableau}
\]
Informally, we think of $h'$ as being ``contained'' in $h$. 
\end{example} 

It is easy to see that  if $h' \subseteq h$, then for any matrix $A$, we have $\Hess(A,h') \subseteq \Hess(A,h)$. In the case of the regular nilpotent Hessenberg variety, the fact that $H^*(\Flags(\C^n))$ surjects onto both $H^*(\Hess(\mathsf{N},h'))$ and $H^*(\Hess(\mathsf{N},h))$ implies that the restriction map 
\begin{equation}\label{eq: restriction for Schubert}
H^*(\Hess(\mathsf{N}, h)) \to H^*(\Hess(\mathsf{N}, h'))
\end{equation}
is surjective. 

We can now give an outline of our argument below. Starting with the largest Hessenberg variety for $h_0 = (n,n,\ldots, n)$ which corresponds to the flag variety itself, and for whose cohomology we already know that there are no linear relations among the Schubert classes, we remove boxes from the Hessenberg diagram of $h_0$ one at a time, analyzing at each step the kernel of the corresponding restriction map~\eqref{eq: restriction for Schubert}. Indeed, we will obtain a basis for the kernel of~\eqref{eq: restriction for Schubert} at each step. Putting them together, we can then obtain a basis for the kernel of the map 
\[
H^*(\Flags(\C^n)) \to H^*(\Hess(\mathsf{N}, h))
\]
for any Hessenberg function $h$, and thus obtain a basis of the linear relations satisfied by the images of the Schubert classes. 

More precisely, we have the following. As discussed above, in our arguments below we will remove a single box from a diagram of a Hessenberg function in order to obtain a smaller Hessenberg function. In order for the resulting diagram to be the diagram of a valid Hessenberg function, we must place some additional hypotheses. Specifically, let $h: [n] \to [n]$ be a Hessenberg function. We will say that the box $(i,j)$ is a \textbf{corner} of the diagram corresponding to $h$ if $h(j) = i$ and if $h(j-1)<h(j)$ (the latter condition is vacuous if $j=1$). 
Note that if box $(i,j)$ is a corner of $h$, then $r_i=j$; we use this fact several times in the proof of Proposition~\ref{proposition: relation on Schuberts for Asj} below.

\begin{example} 
Let $h = (3,4,4,5,5)$. The diagram of $h$ is 
\[\ytableausetup{centertableaux} 
  \; \;   \begin{ytableau}
\star &  \star  & \star & \star & \star \\ 
\star & \star & \star & \star   & \star   \\
 \star &  \star  & \star & \star & \star \\ 
 &  \star & \star  & \star & \star \\ 
 & &  & \star & \star \\ 
\end{ytableau}
\]
and the corners of $h$ are the $(3,1)$-th, $(4,2)$-th and $(5,4)$-th boxes. 
\end{example}

Given $h: [n] \to [n]$ a Hessenberg function and a box $(i,j)$ which is a corner of $h$ with $i>j$, it is clear that if we remove the box $(i,j)$ from the diagram of $h$, we obtain a diagram of a Hessenberg function, which we denote by $h'$. We use this terminology in the theorem below.

\begin{theorem}\label{theorem: relation on Schuberts} 
Let $n$ be a positive integer, $n \geq 2$. 
Let $h: [n] \to [n]$ be a Hessenberg function and suppose that the box $(i,j)$ for $i>j$ is a corner of $h$. Let $h': [n] \to [n]$ be the Hessenberg function obtained from $h$ by removing the box $(i,j)$. Then the kernel of~\eqref{eq: restriction for Schubert} has as a basis the following set: 
\begin{equation}\label{eq: kernel for Schubert} 
\left\{ x_1^{i_1} \cdots x_{j-1}^{i_{j-1}} \cdot f_{i-1, j}  \cdot x_{j+1}^{i_{j+1}} \cdots x_n^{i_n} \, \mid \, 0 \leq i_m \leq h(m)-m \textup{ if } 1 \leq m \leq n, m \neq j \right\}.
\end{equation} 
\end{theorem}

The strategy of our proof will be similar to the one used for 
Theorem~\ref{theorem: monomial basis}.
We first define analogous ring maps $\varphi_s: \mathcal{A}_s^{h} \to \mathcal{A}_s^{h'}$ in such a way that the special case $s=1$ is exactly the restriction map~\eqref{eq: restriction for Schubert}. Specifically, we define 
\begin{equation}\label{eq: def varphis}
\varphi_s: \mathcal{A}_s^{h} \to \mathcal{A}_s^{h'}
\end{equation}
as the ring homomorphisim induced from the identity homomorphism $\mathcal{R} \to \mathcal{R}$. Here it is useful to recall that 
\[
 \mathcal{A}_s^h = \mathcal{R}/\langle g_{h(1),1}, \cdots, g_{h(s-1),s-1}, f_{h(s),s}, \cdots, f_{h(n),n} \rangle 
\]
and
\[
\mathcal{A}_s^{h'} = \mathcal{R}/\langle g_{h'(1),1}, \cdots,  g_{h'(s-1),s-1}, f_{h'(s),s}, \cdots, f_{h'(n),n} \rangle. 
\]
Before proceeding, we need to show that the $\varphi_s$ thus defined are well-defined; this is part of the next Lemma~\ref{lemma: varphis} below.

\begin{lemma}\label{lemma: varphis} 
The map $\varphi_s$ in~\eqref{eq: def varphis} is well-defined and surjective. 
\end{lemma} 

\begin{proof} 
First suppose that $1 \leq j \leq s-1$. In this case 
the two ideals in question differ only in the generators $g_{h(j), j} = g_{i,j}$ and $g_{h'(j),j} = g_{i-1,j}$. We know from Lemma~\ref{lemma:4-2} that $g_{i,j} \in \langle g_{h'(1),1}, \cdots, g_{h'(j-1),j-1}, g_{h'(j),j}= g_{i-1,j} \rangle$ so the ideal 
$\langle g_{h(1),1}, \cdots, g_{i,j}, \cdots, g_{h(s-1),s-1}, f_{h(s),s}, \cdots, f_{h(n),n} \rangle$ is 
contained in the ideal 
\[
\langle g_{h'(1),1}, \cdots, g_{i-1,j}, \cdots, g_{h'(s-1),s-1}, f_{h'(s),s}, \cdots, f_{h'(n),n} \rangle.
\]
 This implies that the identity homomorphism $\mathcal{R} \to \mathcal{R}$ induces a ring homomorphism $\varphi_s: \mathcal{A}_s^h \to \mathcal{A}_s^{h'}$ as desired, and since the ideal for $h$ is contained in the ideal for $h'$, it is surjective. 

Now suppose that $j \geq s$. In this case, the two ideals differ only in the generators $f_{i,j}$ and $f_{i-1,j}$. We know from Lemma~\ref{lemma:4-2} that 
\[
f_{i,j} \in \langle f_{h'(1),1}, \ldots, f_{h'(j-1),j-1}, f_{h'(j),j} = f_{i-1,j} \rangle
\]
 and we know from Lemma~\ref{lemma:4-3} that 
\[
\langle f_{h'(1),1}, \ldots, f_{h'(j-1),j-1}, f_{h'(j),j} = f_{i-1,j} \rangle 
\subseteq \langle g_{h'(1),1}, \cdots, g_{h'(s-1),s-1}, f_{h'(s),s}, \ldots, f_{h'(j),j} \rangle
\]
so by similar arguments as in the previous case we conclude there is a surjective ring homormophism $\varphi_s$ as claimed. 
\end{proof}

The map $\varphi_s$ also naturally induces a surjective map on the quotient rings 
\[
\overline{\varphi}_s: \mathcal{A}_s^h/\langle x_s \rangle \to \mathcal{A}_s^{h'}/\langle x_s \rangle.
\]
Now recall that both rings $\mathcal{A}_s^{h}$ and $\mathcal{A}_s^{h'}$ fit into exact sequences of the form~\eqref{equation:exact}.  
We can put these together with the surjective ring homomorphisms $\varphi_s$ and $\overline{\varphi}_s$ 
to obtain a larger commutative diagram as follows. 

\begin{lemma}\label{lemma: kernel exact seq}
Fix $s$, $1 \leq s \leq n$. Then there exists the following commutative diagram 
\begin{equation}\label{eq: big comm diagram} 
\xymatrix{ 
         & 0 \ar[d] & 0 \ar[d] & 0 \ar[d] &  \\
0 \ar[r] & \ker(\varphi_{s+1}) \ar[r]^{\times x_s} \ar[d] & \ker(\varphi_s) \ar[r] \ar[d] & \ker(\overline{\varphi}_s) \ar[r] \ar[d] & 0\\
0 \ar[r] & \mathcal{A}_{s+1}^{h} \ar[r]^{\times x_s} \ar[d]^{\varphi_{s+1}} & \mathcal{A}_s^{h} \ar[r] \ar[d]^{\varphi_s} & \mathcal{A}_s^{h}/\langle x_s \rangle \ar[r] \ar[d]^{\overline{\varphi}_s} & 0 \\
0 \ar[r] & \mathcal{A}_{s+1}^{h'} \ar[r]^{\times x_s} \ar[d] & \mathcal{A}_s^{h'} \ar[r] \ar[d] & \mathcal{A}_s^{h'}/\langle x_s \rangle \ar[r] \ar[d] & 0 \\
 & 0 & 0 & 0 & \\
}
\end{equation}
where $\overline{\varphi}_s$ is induced from $\varphi_s$ on the corresponding quotient rings. In particular, the sequence 
\begin{equation}\label{eq: exact h s+1 s} 
\xymatrix{ 
0 \ar[r] & \ker(\varphi_{s+1}) \ar[r]^{\times x_s} & \ker(\varphi_s) \ar[r] & \ker(\overline{\varphi}_s) \ar[r] & 0
}
\end{equation}
is exact. 
\end{lemma}

\begin{proof} 
The diagram~\eqref{eq: big comm diagram} is clearly commutative since all maps are induced from the identity map $\mathcal{R} \to \mathcal{R}$ or the multiplication map by $x_s$, and exactness follows from the Snake Lemma. 
\end{proof}

In the proof of Proposition~\ref{prop: first half of main theorem} we use that (up to a certain equivalence) $f_{i,j}$ may be expressed as a product $x_j g_{i,j}$ and apply Lemma~\ref{lemma:key} to obtain an exact sequence by ``peeling off'' the factor $x_j$ from the product $x_j g_{i,j}$. However, there is nothing preventing us from ``peeling off'' the factor $g_{i,j}$ instead, and the next lemma records what happens when we do so.

\begin{lemma}\label{lemma: peel gij} 
Fix $\ell$ such that $1 \leq \ell \leq \mathsf{p}(h)$. There exists a natural ring surjection $\Psi_\ell: \mathcal{A}_\ell^{h} \to \mathcal{A}_{\ell+1}^{h}$ which fits in an exact sequence 
\begin{equation}\label{eq: exact seq opp}
\xymatrix{
0 \ar[r] & \mathcal{A}_\ell^{h}/\langle x_\ell \rangle \ar[rr]^-{\times g_{h(\ell),\ell}} & & \mathcal{A}_\ell^{h} \ar[r]^{\Psi_\ell} & \mathcal{A}_{\ell+1}^{h} \ar[r] & 0.
}
\end{equation}
Moreover, the set of monomials
\[
\left\{ x_1^{i_1} \cdots x_{\ell-1}^{i_{\ell-1}} \cdot g_{h(\ell),\ell} \cdot x_{\ell+1}^{i_{\ell+1}} \cdots x_n^{i_n} \, \mid \, 
 \begin{array}{l} 
 0 \leq i_m \leq h(m)-m-1 \ \ \textup{ if } 1 \leq m \leq \ell-1  \\ 
0 \leq i_m \leq h(m)-m \ \ \textup{ if } \ell+1 \leq m \leq n 
\end{array}
\right\} 
\]
 is an additive basis for the kernel of $\Psi_\ell$. 
\end{lemma}

\begin{proof} 
If $\ell=\mathsf{p}(h)$, then $g_{h(\ell).\ell}=g_{h(\mathsf{p}(h)).\mathsf{p}(h)}=\mathsf{p}(h)$ is a constant. The claim follows from Proposition~\ref{prop: first half of main theorem}.

In what follows, we assume that $1 \leq \ell < \mathsf{p}(h)$.
First observe that the sum ideal  
\[
\langle g_{h(1),1}, \cdots, g_{h(\ell-1),\ell-1}, f_{h(\ell),\ell}, \cdots, f_{h(n),n} \rangle + \langle g_{h(\ell),\ell} \rangle
\]
is equal to the ideal 
\[
\langle g_{h(1),1}, \cdots, g_{h(\ell-1),\ell-1}, g_{h(\ell),\ell}, f_{h(\ell+1),\ell+1}, \cdots, f_{h(n),n} \rangle 
\]
by 
Lemma~\ref{lemma:4-3}.
Also, we know that the ideals 
\[
\langle g_{h(1),1}, \cdots, g_{h(\ell-1),\ell-1}, f_{h(\ell),\ell}, f_{h(\ell+1),\ell+1}, \cdots, f_{h(n),n} \rangle
\]
and 
\[
\langle g_{h(1),1}, \cdots, g_{h(\ell-1),\ell-1}, x_\ell \cdot g_{h(\ell),\ell}, f_{h(\ell+1),\ell+1}, \cdots, f_{h(n),n}\rangle 
\]
are equal by 
Lemma~\ref{lemma: equal ideals}.
We also know by previous arguments that the sequence of generators in the above ideal is a regular sequence Lemma~\ref{lemma:4-4}. Applying Lemma~\ref{lemma:key} using $g_n = x_\ell \cdot g_{h(\ell),\ell}$ and $g''_n = g_{h(\ell),\ell}$ and $g'_n = x_\ell$ yields the desired exact sequence~\eqref{eq: exact seq opp}. In particular, the map $\Psi_\ell$ is the map which quotients by $g_{h(\ell),\ell}$. 

Now recall from Proposition~\ref{proposition:5-1} that 
\[
\mathcal{A}_{\ell}^{h}/\langle x_\ell \rangle \cong \mathcal{A}_{r_{\ell}}^{h^{(\ell)}}
\]
and from Theorem~\ref{theorem: monomial basis} we know that the set of monomials 
\[
\left\{ y_1^{i_1} \cdots y_{r_\ell-1}^{i_{r_\ell-1}} y_{r_\ell}^{i_{r_{\ell}}} \cdots y_{n-1}^{i_{n-1}} \, \mid \, 
\begin{array}{l} 
0 \leq i_m \leq h^{(\ell)}(m)-m-1 \ \ \textup{ if } 1 \leq m \leq r_{\ell}-1  \\
0 \leq i_m \leq h^{(\ell)}(m) -m \ \ \textup{ if } r_\ell \leq m \leq n-1 
\end{array} 
\right\} 
\]
is an additive basis of $\mathcal{A}_{r_\ell}^{h^{(\ell)}}$. Using the same isomorphism $\varphi : \mathcal{A}_{\ell}^{h}/\langle x_\ell \rangle \to \mathcal{A}_{r_{\ell}}^{h^{(\ell)}}$ as in the proof of Proposition~\ref{proposition:5-1} together with the definition of $h^{(\ell)}$ in~\eqref{eq: def hs}, it is straightforward to see that the above monomials correspond under $\varphi$ to the monomials 
\[
\left\{ x_1^{i_1} \cdots x_{\ell-1}^{i_{\ell-1}} x_{\ell+1}^{i_{\ell+1}} \cdots x_n^{i_n} \, \mid \, 
\begin{array}{l} 
0 \leq i_m \leq h(m)-m-1 \ \ \textup{ if } 1 \leq m \leq \ell-1 \\
0 \leq i_m \leq h(m)-m \ \ \textup{ if } \ell+1 \leq m \leq n 
\end{array} 
\right\} 
\]
By the exactness of the sequence~\eqref{eq: exact seq opp}, the image of these basis elements under the map which multiplies by $g_{h(\ell),\ell}$ is the kernel of $\Psi_\ell$, so the result follows. 
\end{proof}

We now state and prove a proposition which is a generalization of Theorem~\ref{theorem: relation on Schuberts}. In particular, we can obtain Theorem~\ref{theorem: relation on Schuberts} applying the proposition below to the case $s=1$. 

\begin{proposition}\label{proposition: relation on Schuberts for Asj} 
Let $n$ be a positive integer, $n \geq 2$. 
Let $h: [n] \to [n]$ be a Hessenberg function and suppose that the box $(i,j)$ for $i>j$ is a corner of $h$. Let $h': [n] \to [n]$ be the Hessenberg function obtained from $h$ by removing the box $(i,j)$. 
Let $1 \leq s \leq \mathsf{p}(h)$. If $j+1 \leq s$ then the set of monomials 
\begin{equation}\label{eq: relations Asj g}
\left\{ x_1^{i_1} \cdots x_{j-1}^{i_{j-1}} \cdot g_{i-1,j} \cdot x_{j+1}^{i_{j+1}} \cdots x_{s-1}^{i_{s-1}} x_s^{i_s} \cdots x_n^{i_n} \, \mid \, 
\begin{array}{l} 
0 \leq i_m \leq h(m)-m-1 \ \ \textup{ if } 1 \leq m \leq s-1, m \neq j \\
0 \leq i_m \leq h(m)-m \ \ \textup{ if } s \leq m \leq n 
\end{array} 
\right\} 
\end{equation}
is an additive basis of $\ker(\varphi_s) \subseteq \mathcal{A}_s^h$. If $s\leq j$, then the set of monomials 
\begin{equation}\label{eq: relations Asj f}
\left\{ x_1^{i_1} \cdots x_{s-1}^{i_{s-1}} x_s^{i_s} \cdots x_{j-1}^{i_{j-1}} \cdot f_{i-1,j} \cdot x_{j+1}^{i_{j+1}} \cdots x_n^{i_n} \, \mid \, 
\begin{array}{l} 
0 \leq i_m \leq h(m)-m-1 \ \ \textup{ if } 1 \leq m \leq s-1 \\
0 \leq i_m \leq h(m)-m \ \ \textup{ if } s \leq m \leq n, m \neq j 
\end{array} 
\right\} 
\end{equation}
is an additive basis of $\ker(\varphi_s) \subseteq \mathcal{A}_s^h$. 
\end{proposition}

\begin{proof} 
We prove the claim by induction on $n$. At each inductive step (with respect to $n$) we also use a decreasing induction on $s$. 

The base case is $n=2$. Here the only possible choice of Hessenberg function $h$ to which the statement of the proposition can be applied is $h=(2,2)$, where the choice of box to be removed is the $(2,1)$-th box. Thus the Hessenberg function $h'$, obtained by removing the $(2,1)$ box from $h$, is $h'=(1,2)$. Here $\mathsf{p}(h)=2$, so there are two cases to consider $s=1$ and $s=2$. The corner being removed is $(2,1)$, so $i=2$ and $j=1$. Suppose $s=1$. Then we have
\[
\mathcal{A}_1^{h} = \Q[x_1,x_2]/ \langle (x_1-x_2)x_1, x_1+x_2 \rangle \cong \Q[x_1]/\langle x_1^2 \rangle
\]
and 
\[
\mathcal{A}_1^{h'} = \Q[x_1,x_2]/ \langle x_1, x_1+x_2 \rangle \cong \Q.
\]
In this case we have $s=1\leq j=1$ so the statement of the proposition claims that $\{f_{1,1} \cdot x_2^{i_2} \, \mid \, 0 \leq i_2 \leq h(2)-2=0 \} = \{f_{1,1} \} = \{x_1\}$ is a basis of the kernel of $\varphi_1$. This can easily be checked since $\varphi_1$ is the map which sends $x_1$ to $0$ and $1$ to $1$. 

Now suppose $s=2$. Then 
\[
\mathcal{A}_2^h = \Q[x_1,x_2]/\langle g_{2,1}, f_{2,2} \rangle = \Q[x_1, x_2]/\langle x_1-x_2, x_1+x_2 \rangle \cong \Q
\]
and 
\[
\mathcal{A}_2^{h'} = \Q[x_1,x_2]/\langle g_{1,1}, f_{2,2} \rangle \cong \{0\}.
\]
Here the map $\varphi_2$ is the map sending $1$ to $0$. Since $j+1=2 \leq s=2$ we need to check that 
$\{ g_{1,1} \cdot x_2^{i_2} \, \mid \, 0 \leq i_2 \leq h(2)-2 = 0 \} = \{1\}$ is a basis of the kernel. This follows immediately from the computation of $\varphi_2$. This completes the base case of $n=2$. 

Now assume that the assertion of the proposition holds for $n-1$, with any allowable choices of $h$, box $(i,j)$ and $s$. We now prove that the assertion holds for $n$, using a descending induction on $s$. 

Let $h: [n] \to [n]$, the box $(i,j)$ and the Hessenberg function $h'$ be as in the statement of the proposition. We will make the induction argument separately for the two cases $\mathsf{p}(h) > j$ and $\mathsf{p}(h) < j$. (Note that $\mathsf{p}(h)=j$ cannot occur since we assume $h(j)=i>j$ so $h(j) \neq j$.)

\smallskip
\noindent \textbf{Case (a): $\mathsf{p}(h)>j$.} 
\smallskip

In this case, notice that $i \leq \mathsf{p}(h)$ because $j < \mathsf{p}(h)$ and Hessenberg functions are non-decreasing by assumption, so $i=h(j) \leq h(\mathsf{p}(h)) = \mathsf{p}(h)$. Thus we have that 
\[
1 \leq j < i \leq \mathsf{p}(h).
\]
The parameter $s$ is required to satisfy $1 \leq s \leq \mathsf{p}(h)$, so the base case of the 
decreasing induction is when $s=\mathsf{p}(h)$. 
The argument that follows has many cases and is rather long, so we first give a sketch of how the argument proceeds. 
By the short exact sequence from Lemma~\ref{lemma: kernel exact seq}
\begin{equation}\label{eq: SES}
0 \to \ker(\varphi_{s+1}) \xrightarrow{\times x_s} \ker(\varphi_s) \to \ker (\overline
\varphi_s) \to 0
\end{equation}
the union of a basis of $\ker(\varphi_{s+1})$ multiplied by $x_s$ and a basis of $\ker
(\overline \varphi_s)$ is a basis of $\ker(\varphi_s)$. Our arguments below use this fact 
together with the induction hypotheses on both $n$ and $s$ applied to $\ker(\overline \varphi_s)$ and 
$\ker(\varphi_{s+1})$ respectively. 
We also repeatedly use the identification 
\begin{equation}\label{eq: induction identification} 
\mathcal{A}^h_s/\langle x_s \rangle \cong \mathcal{A}^{h^{(s)}}_{r_s}.
\end{equation}
Finally, we use the notation $r'_s$ to denote the RHS of~\eqref{eq: def rs} for the Hessenberg function $h'$. It is immediate from the construction of $h'$ that $r_s=r'_s$ unless $s=i$; we use this repeatedly as well. 
Details of the arguments differ slightly depending on the case under consideration.

\smallskip 
\textbf{Case (a-1) (base case of descending induction on $s$): $s = \mathsf{p}(h)$. } 

\smallskip
 From Lemma~\ref{lemma: kernel exact seq} we know there is an exact sequence~\eqref{eq: SES}. 
On the other hand, since $\mathcal{A}^h_{\mathsf{p}(h)+1} = 0$ it follows that $\ker(\varphi_{\mathsf{p}(h)+1}) = 0$, so we have $\ker(\varphi_{\mathsf{p}(h)}) \cong \ker(\overline{\varphi_{\mathsf{p}(h)}})$. Next recall that $\mathcal{A}^h_{\mathsf{p}(h)}/\langle x_{\mathsf{p}(h)} \rangle \cong \mathcal{A}^{h^{(\mathsf{p}(h))}}_{r_{\mathsf{p}(h)}}$ by Proposition~\ref{proposition:5-1} and similarly 
$\mathcal{A}^{h'}_{\mathsf{p}(h)}/\langle x_{\mathsf{p}(h)} \rangle \cong \mathcal{A}^{h'^{(\mathsf{p}(h))}}_{r'_{\mathsf{p}(h)}}$. 
Now we take cases again.

\smallskip
\textbf{Case (a-1-(i)):} $i=s=\mathsf{p}(h)$. In this case we can see that $j =r_i =r_{\mathsf{p}(h)}$ (since $(i,j)$ must be a corner) and it follows that $h'^{(\mathsf{p}(h))} = h^{(\mathsf{p}(h))}$ and $r'_{\mathsf{p}(h)} = r_{\mathsf{p}(h)}+1=j+1$. Thus the map $\overline{\varphi_{\mathsf{p}(h)}}$ can be viewed as a map 
\[
\overline{\varphi_{\mathsf{p}(h)}}: \mathcal{A}^{h^{(\mathsf{p}(h))}}_j \to \mathcal{A}^{h^{(\mathsf{p}(h))}}_{j+1}
\]
and it is also the same as the homomorphism $\Psi_j$ considered in Lemma~\ref{lemma: peel gij} since $j \leq \mathsf{p}(h^{(\mathsf{p}(h))})$ by Lemma~\ref{lemma: induction step}. In particular, from Lemma~\ref{lemma: peel gij} we conclude that $\ker(\overline{\varphi_{\mathsf{p}(h)}}) = \ker(\Psi_{j})$ has a basis 
\begin{equation}\label{eq: y monomials case a-1-i}
\left\{ y_1^{i_1} \cdots y_{j-1}^{i_{j-1}} \cdot g_{h^{(\mathsf{p}(h))}(j), j}(y) \cdot y_{j+1}^{i_{j+1}} \cdots y_{n-1}^{i_{n-1}} \, \left\lvert \, 
\begin{array}{l} 
0 \leq i_m \leq h^{(\mathsf{p}(h))}(m)-m-1 \ \  \textup{ if } 1 \leq m \leq j-1 \\ 
0 \leq i_m \leq h^{(\mathsf{p}(h))}(m)-m \ \ \textup{ if } j+1 \leq m \leq n-1
\end{array} \right. 
\right\}. 
\end{equation}
The isomorphism $\mathcal{A}^h_{\mathsf{p}(h)}/\langle x_{\mathsf{p}(h)} \rangle \cong \mathcal{A}^{h^{(\mathsf{p}(h))}}_{r_{\mathsf{p}(h)}}$ of~\eqref{eq: induction identification} identifies $y_m$ with $x_m$ for $1 \leq m \leq \mathsf{p}(h)-1$ and $y_m$ with $x_{m+1}$ for $\mathsf{p}(h) \leq m \leq n-1$. 
Since $j < \mathsf{p}(h)$, we have $h^{(\mathsf{p}(h))}(j)=i-1$. This means that $g_{h^{(\mathsf{p}(h))}(j), j}(y)=g_{i-1,j}(y)$ is identified with $g_{i-1, j}(x)$.
Moreover, from the definition of $h^{(\mathsf{p}(h))}$ in~\eqref{eq: def h ph m}, it follows that the monomials in~\eqref{eq: y monomials case a-1-i} are identified with the monomials 
\[
\left\{ x_1^{i_1} \cdots x_{j-1}^{i_{j-1}} \cdot g_{i-1, j}(x) \cdot x_{j+1}^{i_{j+1}} \cdots x_{\mathsf{p}(h)-1}^{i_{\mathsf{p}(h)-1}} \cdot x_{\mathsf{p}(h)+1}^{i_{\mathsf{p}(h)+1}} \cdots x_n^{i_n}  \, \left \lvert \,  
\begin{array}{l}
0 \leq i_m \leq h(m)-m-1 \ \ \textup{ if } 1 \leq m \leq \mathsf{p}(h)-1 \\ 
\phantom{0 \leq i_m \leq h^{(\mathsf{p}(h))}(m)-m-1} \  \  \textup{ and } m \neq j \\  
0 \leq i_m \leq h(m)-m \ \ \textup{ if } \mathsf{p}(h)+1 \leq m \leq n 
\end{array} \right.
\right\}. 
\]
Recall that $s=\mathsf{p}(h)$ in this case, and $h(\mathsf{p}(h)) = \mathsf{p}(h)$ by definition, so the condition $0 \leq i_s \leq h(s)-s$ implies $i_s = 0$ and hence $x_s$ does not appear in the monomials in the set~\eqref{eq: relations Asj g} and it can be seen that the monomials in the equation above are exactly those in~\eqref{eq: relations Asj g}, as desired. This completes the case (a-1-(i)). 

\smallskip
\textbf{Case (a-1-(ii)):} $i < \mathsf{p}(h)$. 

In this case, note first that $j<r_{\mathsf{p}(h)}$ since $h(j)=i<\mathsf{p}(h)$. Moreover, because $h$ and $h'$ differ by a box with $i<\mathsf{p}(h)$ and $j < \mathsf{p}(h)$, it follows that $r'_{\mathsf{p}(h)} = r_{\mathsf{p}(h)}$ and that $(h')^{(\mathsf{p}(h))}$ is the Hessenberg function obtained by removing the box $(i,j)$ from $h^{\mathsf{p}(h)}$. Thus, from the inductive hypothesis on $n$ it follows that the kernel of the map 
\[
\overline{\varphi_{\mathsf{p}(h)}}: \mathcal{A}^{h^{(\mathsf{p}(h))}}_{r_{\mathsf{p}(h)}} \to \mathcal{A}^{(h')^{(\mathsf{p}(h))}}_{r'_{\mathsf{p}(h)}} = \mathcal{A}^{(h')^{(\mathsf{p}(h))}}_{r_{\mathsf{p}(h)}}
\]
has as an additive basis the following set of monomials: 
\[
\left\{ y_1^{i_1} \cdots y_{j-1}^{i_{j-1}} \cdot g_{i-1, j}(y) \cdot y_{j+1}^{i_{j+1}} \cdots
y_{r_{\mathsf{p}(h)-1}}^{i_{r_{\mathsf{p}(h)-1}}} y_{r_{\mathsf{p}(h)}}^{i_{r_{\mathsf{p}(h)}}} \cdots 
 y_{n-1}^{i_{n-1}} \, \left\lvert \, 
\begin{array}{l} 
0 \leq i_m \leq h^{(\mathsf{p}(h))}(m)-m-1 \ \  \textup{ if } 1 \leq m \leq r_{\mathsf{p}(h)}-1   \\ 
\phantom{0 \leq i_m \leq h^{(\mathsf{p}(h))}(m)-m-1} \  \  \textup{ and } m \neq j \\ 
0 \leq i_m \leq h^{(\mathsf{p}(h))}(m)-m \ \ \textup{ if } r_{\mathsf{p}(h)} \leq m \leq n-1
\end{array} \right.
\right\}. 
\]
Using the same isomorphism $\mathcal{A}^h_{\mathsf{p}(h)}/\langle x_{\mathsf{p}(h)} \rangle \cong \mathcal{A}^{h^{(\mathsf{p}(h))}}_{r_{\mathsf{p}(h)}}$ of~\eqref{eq: induction identification} as in the case (a-1-(i)) it can be seen that the above monomials are identified with 
\[
\left\{ x_1^{i_1} \cdots x_{j-1}^{i_{j-1}} \cdot g_{i-1, j}(x) \cdot x_{j+1}^{i_{j+1}} \cdots x_{\mathsf{p}(h)-1}^{i_{\mathsf{p}(h)-1}} \cdot x_{\mathsf{p}(h)+1}^{i_{\mathsf{p}(h)+1}} \cdots x_n^{i_n} \, \left\lvert \, 
\begin{array}{l}
0 \leq i_m \leq h(m)-m-1 \ \ \textup{ if } 1 \leq m \leq \mathsf{p}(h)-1  \\ 

\phantom{0 \leq i_m \leq h^{(\mathsf{p}(h))}(m)-m-1} \ \ \textup{ and } m \neq j \\ 
0 \leq i_m \leq h(m)-m \ \ \textup{ if } \mathsf{p}(h)+1 \leq m \leq n 
\end{array} \right.
\right\}
\]
which is the set of monomials appearing in~\eqref{eq: relations Asj g}, by an argument similar to the case (a-1-(i)). 
This completes case (a-1), the base case of the descending induction on $s$. 
\smallskip

Going forward we assume by induction that the result is known for higher values of $s$. 

\smallskip
\noindent \textbf{Case (a-2):} $i < s < \mathsf{p}(h)$.

Consider the exact sequence of~\eqref{eq: SES}. 
From the assumptions it follows that $j < r_s\leq s$, $r'_s = r_s$ and that $(h')^{(s)}$ is the Hessenberg function obtained by removing the box $(i,j)$ from $h^{(s)}$. Hence from the inductive hypothesis on $n$ we know that the kernel of 
\[
\overline{\varphi_s}: \mathcal{A}^{h^{(s)}}_{r_s} \to \mathcal{A}^{(h')^{(s)}}_{r'_s} =  \mathcal{A}^{(h')^{(s)}}_{r_s}
\]
has, as an additive basis, the set of monomials 
\[
\left\{ y_1^{i_1} \cdots y_{j-1}^{i_{j-1}} \cdot g_{i-1, j}(y) \cdot y_{j+1}^{i_{j+1}} \cdots
y_{r_{s-1}}^{i_{r_{s-1}}} y_{r_{s}}^{i_{r_{s}}} \cdots 
 y_{n-1}^{i_{n-1}} \, \left\lvert \, 
\begin{array}{l} 
0 \leq i_m \leq h^{(s)}(m)-m-1 \ \  \textup{ if } 1 \leq m \leq r_{s}-1 \\
\phantom{0 \leq i_m \leq h^{(\mathsf{p}(h))}(m)-m-1} \ \ \textup{ and } m \neq j \\ 
0 \leq i_m \leq h^{(s)}(m)-m \ \ \textup{ if } r_{s} \leq m \leq n-1
\end{array} \right.
\right\}. 
\]
The isomorphism $\mathcal{A}^h_{s}/\langle x_{s} \rangle \cong \mathcal{A}^{h^{(s)}}_{r_{s}}$ of~\eqref{eq: induction identification} identifies 
$y_m$ with $x_m$ for $1 \leq m \leq s-1$ and $y_m$ with $x_{m+1}$ for $s \leq m \leq n-1$. By similar considerations as in the previous cases, it follows that the above monomials are identified with the monomials 
\begin{equation}\label{eq: n-1 monomials a-2} 
\left\{ x_1^{i_1} \cdots x_{j-1}^{i_{j-1}} \cdot g_{i-1, j}(x) \cdot x_{j+1}^{i_{j+1}} \cdots x_{s-1}^{i_{s-1}} \cdot x_{s+1}^{i_{s+1}} \cdots x_n^{i_n} \, \left\lvert \, 
\begin{array}{l}
0 \leq i_m \leq h(m)-m-1 \ \ \textup{ if } 1 \leq m \leq s-1  \\ 
\phantom{0 \leq i_m \leq h(m)-m-1} \ \ \textup{ and } m \neq j \\ 
0 \leq i_m \leq h(m)-m \ \ \textup{ if } s+1 \leq m \leq n 
\end{array} \right.
\right\}.
\end{equation}
Furthermore, by the inductive hypothesis on $s$ we know that $\ker(\varphi_{s+1})$ has the following set of monomials 
\[
\left\{ x_1^{i_1} \cdots x_{j-1}^{i_{j-1}} \cdot g_{i-1, j}(x) \cdot x_{j+1}^{i_{j+1}} \cdots x_{s}^{i_{s}} \cdot x_{s+1}^{i_{s+1}} \cdots x_n^{i_n} \, \left\lvert \, 
\begin{array}{l}
0 \leq i_m \leq h(m)-m-1 \ \ \textup{ if } 1 \leq m \leq s \\
\phantom{0 \leq i_m \leq h(m)-m-1} \ \ \textup{ and } m \neq j \\ 
0 \leq i_m \leq h(m)-m \ \ \textup{ if } s+1 \leq m \leq n 
\end{array} \right.
\right\}
\]
as an additive basis. The map $\ker(\varphi_{s+1}) \to \ker(\varphi_s)$ multiplies these monomials by the variable $x_s$, so the image of this set in $\ker(\varphi_s)$ is 
\begin{equation}\label{eq: s+1 monomials a-2}
\left\{ x_1^{i_1} \cdots x_{j-1}^{i_{j-1}} \cdot g_{i-1, j}(x) \cdot x_{j+1}^{i_{j+1}} \cdots x_{s}^{i_{s}} \cdot x_{s+1}^{i_{s+1}} \cdots x_n^{i_n} \, \left\lvert \, 
\begin{array}{l}
0 \leq i_m \leq h(m)-m-1 \ \ \textup{ if } 1 \leq m \leq s-1 \\
\phantom{0 \leq i_m \leq h(m)-m-1} \ \ \textup{ and } m \neq j \\ 
1 \leq i_s \leq h(s)-s \ \ \textup{ if } m=s \\ 
0 \leq i_m \leq h(m)-m \ \ \textup{ if } s+1 \leq m \leq n 
\end{array} \right.
\right\}.
\end{equation}
By the exactness of the sequence~\eqref{eq: SES} we know that the union of~\eqref{eq: s+1 monomials a-2}
and~\eqref{eq: n-1 monomials a-2} is an additive basis of $\ker(\varphi_s)$, and this set coincides with~\eqref{eq: relations Asj g} as desired. 

\smallskip
\noindent \textbf{Case (a-3):} $s=i$. 
\smallskip

In this case we have $h^{(s)} = (h')^{(s)}$ and $r_s = j$ and $r'_s = j+1$. From the same exact sequence as 
in~\eqref{eq: SES}
we wish to use the inductive hypotheses to obtain monomial bases for $\ker(\overline{\varphi}_s)$ and $\ker(\varphi_{s+1})$. In this case an additive basis of the kernel of the map $\overline{\varphi}_s$ can be obtained using Lemma~\ref{lemma: peel gij} as in the case (a-1-(i)) and a basis for the kernel of $\varphi_{s+1}$ can be described by induction on $s$ as in case (a-2) above. A similar argument using~\eqref{eq: SES} yields the result in this case. 

\smallskip
\noindent \textbf{Case (a-4):} $j < s < i$. 
\smallskip

In this case, it can be seen from the fact that $h(j)=i >s$ and the fact that $(i,j)$ is a corner, that $r_s \leq j$. Moreover, $(h')^{(s)}$ is the Hessenberg function obtained from $h^{(s)}$ by removing the $(i-1,j)$-th box (note that since $s<i$, the $s$-th row, which gets removed in $h^{(s)}$, lies above the $i$-th row). From these considerations we see that the map 
\[
\overline{\varphi}_s: \mathcal{A}_{r_s}^{h^{(s)}} \to \mathcal{A}_{r'_s = r_s}^{(h')^{(s)}}
\]
has a kernel which can be described by the induction hypothesis on $n$. Specifically, $\ker(\overline{\varphi}_s)$ has a basis consisting of the monomials 
\begin{equation}\label{eq: monomials a-4}
\left\{ 
y_1^{i_1} \cdots y_{r_s-1}^{i_{r_s-1}} y_{r_s}^{i_{r_s}} \cdots y_{j-1}^{i_{j-1}} \cdot f_{(i-1)-1,j}(y) \cdot y_{j+1}^{i_{j+1}} \cdots y_{n-1}^{i_{n-1}} \, \left\lvert \, 
\begin{array}{l} 
0 \leq i_m \leq h^{(s)}(m)-m-1 \ \ \textup{ if } 1 \leq m \leq r_s-1 \\
0 \leq i_m \leq h^{(s)}(m)-m \ \ \textup{ if } r_s \leq m \leq n-1 \\
\phantom{0 \leq i_m \leq h(m)-m-1} \ \ \textup{ and } m \neq j \\ 
\end{array} \right.
\right\}
\end{equation}
where the inequality $r_s \leq j$ implies that we use $f_{(i-1)-1,j} = f_{i-2,j}$ instead of $g_{i-2,j}$ in the expressions above. 

Here we note that $f_{i-2,j}(y) = \sum_{k=1}^j \left( \prod_{\ell=j+1}^{i-2} (y_k - y_\ell) \right) y_k$ is identified with 
\[
\sum_{k=1}^j \left( \prod_{\ell=j+1}^{s-1} (x_k - x_\ell) \right) x_k + 
\sum_{k=1}^j \left( \prod_{\ell=s}^{i-2} (x_k - x_{\ell+1}) \right) x_k 
= 
\sum_{k=1}^j \left( \prod_{\ell=j+1, \ell \neq s}^{i-1} (x_k - x_\ell) \right) x_k 
\]
under the same isomorphism as in the cases above. On the other hand, we have 
\[
g_{i-1,j}(x) = \sum_{k=1}^j \left( \prod_{\ell=j+1}^{i-1} (x_k - x_\ell) \right) \equiv 
\sum_{k=1}^j \left( \prod_{\ell =j+1, \ell \neq s}^{i-1} (x_k - x_\ell) \right) \cdot x_k  \ \ \textup{ (modulo $x_s$) } 
\]
and since $\mathcal{A}_{r_s}^{h^{(s)}} \cong \mathcal{A}_s^{h}/\langle x_s \rangle$ quotients by $x_s$ we conclude that under this isomorphism, the monomials~\eqref{eq: monomials a-4} are identified with 
\[
\left\{ 
x_1^{i_1} \cdots x_{j-1}^{i_{j-1}} \cdot g_{i-1, j}(x) \cdot x_{j+1}^{i_{j+1}} \cdots x_{s-1}^{i_{s-1}} x_{s+1}^{i_{s+1}} \cdots x_n^{i_n} \, \left\lvert \, 
\begin{array}{l} 
0 \leq i_m \leq h(m)-m-1 \ \ \textup{ if } 1 \leq m \leq s-1 \\ 
0 \leq i_m \leq h(m)-m \ \ \textup{ if } s+1 \leq m \leq n 
\end{array} \right.
\right\}.  
\]
The kernel of $\varphi_{s+1}$ can be obtained using the descending induction hypothesis on $s$, and from here the remainder of the argument is as in the cases above. 

\smallskip
\noindent \textbf{Case (a-5):} $s=j$. 
\smallskip

In this case, since $s=j$, the $j$-th row and column are removed from $h$ to create $h^{(s)}$, which means that (since $h$ and $h'$ only differ in the $j$-th column) $h^{(s)} = (h')^{(s)}$. Also, since $i>j$ it follows that $r_s = r'_s$. Thus in this case we have $\ker(\overline{\varphi}_j) = 0$ and we have $\ker(\varphi_{j+1}) \cong \ker(\varphi_j)$ via the isomorphism which multiplies by $x_j$. By our descending induction hypothesis on $s$, we can take as a basis of $\ker(\varphi_{j+1})$ the monomials 
\[
\left\{ 
x_1^{i_1} \cdots x_{j-1}^{i_{j-1}} \cdot g_{i-1, j}(x) \cdot x_{j+1}^{i_{j+1}} \cdots x_n^{i_n} \, \left\lvert \, 
\begin{array}{l} 
0 \leq i_m \leq h(m)-m-1 \ \ \textup{ if } 1 \leq m \leq j-1 \\ 
0 \leq i_m \leq h(m)-m \ \ \textup{ if } j+1 \leq m \leq n 
\end{array} \right.
\right\}.  
\]
We thus obtain a basis of $\ker(\varphi_j)$ by multiplying by $x_j$. However, we also note that $f_{i-1,j} \equiv x_j \cdot g_{i-1,j}$ modulo $g_{i-1, j-1}$ by Lemma~\ref{lemma:4-1}, and $h(j-1) \leq i-1$ since $(i,j)$ is a corner. Thus 
$g_{i-1,j-1} \in \langle g_{h(1),1}, \ldots, g_{h(j-1),j-1} \rangle$ by Lemma~\ref{lemma:4-2} and we conclude $f_{i-1,j} = x_j \cdot g_{i-1,j}$ in $\mathcal{A}_j^h$. Thus we may replace $x_j \cdot g_{i-1,j}$ in the expressions for the monomials and we obtain that the following monomials 
\[
\left\{ 
x_1^{i_1} \cdots x_{j-1}^{i_{j-1}} \cdot f_{i-1,j}(x) \cdot x_{j+1}^{i_{j+1}} \cdots x_n^{i_n} \, \left\lvert \, 
\begin{array}{l} 
0 \leq i_m \leq h(m)-m-1 \ \ \textup{ if } 1 \leq m \leq j-1 \\ 
0 \leq i_m \leq h(m)-m \ \ \textup{ if } j+1 \leq m \leq n 
\end{array} \right.
\right\}
\]
are a basis for $\ker(\varphi_j)$, as desired. 

\smallskip
\noindent \textbf{Case (a-6):} $1 \leq s < j$. 
\smallskip

By assumption, $j<i=h(j)$ so in this case we have $s<j<i$. This implies that $(h')^{(s)}$ is the Hessenberg function obtained from $h^{(s)}$ by removing the $(i-1,j-1)$-st box, and it also implies that $r_s = r'_s$. Moreover, since $r_s \leq s$ and $s<j$, we have $r_s < j$.  Thus the kernel of $\overline{\varphi}_s: \mathcal{A}_{r_s}^{h^{(s)}}\to \mathcal{A}_{r'_s = r_s}^{(h')^{(s)}}$ can be described using the inductive hypothesis on $n$ and we obtain that the set of monomials 
\[
\left\{ y_1^{i_1} \cdots y_{r_s-1}^{i_{r_s-1}} y_{r_s}^{i_{r_s}} \cdots y_{j-2}^{i_{j-2}} \cdot f_{i-2, j-1}(y) \cdot y_j^{i_j} \cdots y_{n-1}^{i_{n-1}} \, \left\lvert \, 
\begin{array}{l} 
0 \leq i_m \leq h^{(s)}(m)-m-1 \ \ \textup{ if } 1 \leq m \leq r_s - 1 \\
0 \leq i_m \leq h^{(s)}(m)-m \ \ \textup{ if } r_s \leq m \leq n-1  
\end{array} \right.  
\right\}
\]
is a basis of $\ker(\overline{\varphi}_s)$. Under the identification of $y_m$ with $x_m$ for $1 \leq m \leq s-1$ and $y_m$ with $x_{m+1}$ for $s \leq m \leq n-1$ in the isomorphism~\eqref{eq: induction identification} we can compute that 
\[
f_{i-2, j-1}(y) = \sum_{k=1}^{j-1} \left( \prod_{\ell=j}^{i-2} (y_k - y_\ell) \right) y_k
\]
is identified with 
\begin{equation*}
\begin{split} 
& \sum_{k=1}^{s-1} \left(\prod_{\ell=j}^{i-2} (x_k - x_{\ell+1})\right) x_k + 
\sum_{k=s}^{j-1} \left( \prod_{\ell=j}^{i-2} (x_{k+1} - x_{\ell+1})\right) x_{k+1} \\
& = 
\sum_{k=1}^{s-1} \left(\prod_{\ell=j+1}^{i-1} (x_k - x_{\ell})\right) x_k 
 +  \sum_{k=s}^{j-1} \left( \prod_{\ell=j+1}^{i-1} (x_{k+1} - x_{\ell}) \right) x_{k+1} \\ 
& = \sum_{k=1}^{s-1} \left(\prod_{\ell=j+1}^{i-1} (x_k - x_{\ell})\right) x_k
+ \sum_{k=s+1}^{j} \left( \prod_{\ell=j+1}^{i-1} (x_{k} - x_{\ell}) \right) x_{k} \\ 
& \equiv \sum_{k=1}^j \left( \prod_{\ell=j+1}^{i-1} (x_{k} - x_{\ell}) \right) x_{k} \ \ \textup{ modulo $x_s$ } \\
& = f_{i-1,j}(x).
\end{split} 
\end{equation*}
Since the ring $\mathcal{A}_s^{h}/\langle x_s \rangle$ quotients by $x_s$, this implies that as a basis of $\ker(\overline{\varphi}_s)$ we may take the monomials 
\[
\left\{ 
x_1^{i_1} \cdots x_{s-1}^{i_{s-1}} x_{s+1}^{i_{s+1}} \cdots x_{j-1}^{i_{j-1}} \cdot f_{i-1,j}(x) \cdot x_{j+1}^{i_{j+1}} \cdots x_n^{i_n}  \, \left\lvert \, 
\begin{array}{l} 
0 \leq i_m \leq h(m)-m-1 \ \ \textup{ if } 1 \leq m \leq s-1 \\
0 \leq i_m \leq h(m)-m \ \ \textup{ if } s+1 \leq m \leq n, m \neq j 
\end{array} \right. 
\right\}. 
\]
 The kernel of $\varphi_{s+1}$ has a basis which may be described by the descending hypothesis on $s$, and the rest of the argument proceeds in a similar manner to the above cases. This completes the arguments for Case (a). 
 
 The argument for Case (b) for $\mathsf{p}(h)<j$ also uses a descending induction on $s$. The argument is similar to that for the case (a-6), so we omit details (note that for the base case $s=\mathsf{p}(h)$, the argument is in fact simpler than case (a-6) because the kernel of $\varphi_{s+1}$ is $0$). 
\end{proof}

The main theorem of this section follows immediately as a corollary. 

\begin{proof}[Proof of Theorem~\ref{theorem: relation on Schuberts}]
It suffices to apply Proposition~\ref{proposition: relation on Schuberts for Asj} to the special case $s=1$, since the rings $\mathcal{A}_1^h$ and $\mathcal{A}_1^{h'}$ are isomorphic to $H^*(\Hess(\mathsf{N},h))$ and $H^*(\Hess(\mathsf{N},h'))$ respectively. 
\end{proof}

In order to complete the argument that derives the linear relations on the Schubert classes in $H^*(\Hess(\mathsf{N},h))$, we need a previous result of the second author. For $w \in \S_n$ a permutation, we let $\mathfrak{S}_w$ denote the corresponding Schubert polynomial. The following result says that the polynomials $f_{i-1,j}$ in this manuscript can be written as an alternating sum of Schubert polynomials.

\begin{theorem}\label{theorem: alternating Schuberts}(\cite[Theorem 1.1]{Horiguchi}) 
 Let $i, j$ be positive integers with $1 \leq j < i \leq n$.  
 Then 
 \[
f_{i-1,j} = \sum_{k=1}^{i-j} (-1)^{k-1} \mathfrak{S}_{w_k^{(i,j)}}
\]
where $w_k^{(i,j)}$ is the permutation in $\S_n$ defined by
\[
w_k^{(i,j)} := (s_{i-k}s_{i-k-1} \cdots s_j )(s_{i-k+1}s_{i-k+2} \cdots s_{i-1})
\]
and $s_r$ denotes the transposition of $r$ and $r + 1$ for any $r$ with $1 \leq r \leq n-1$ and we take 
the convention that $(s_{i-k+1}s_{i-k+2} \cdots s_{i-1}) = e$, where $e$ denotes the identity element, whenever $k = 1$. \end{theorem} 

\begin{remark} 
 The permutation $w_k^{(i,j)}$ appearing in Theorem~\ref{theorem: alternating Schuberts} can also be described in one-line notation as follows:
\[
w_k^{(i,j)} = 1 2 \cdots j-1 \ i-k+1 \ j \ j+1 \cdots \widehat{i-k} \ \widehat{i-k+1} \cdots i \ i-k \ i+1 \cdots n,  
\]
where the caret sign \ $\widehat{}$ \ over an integer $p$ means that the $p$ is to be omitted.
In other words, the permutation $w_k^{(i,j)}$ sends the set $[n] \setminus \{j, j+1, \ldots, i \}$ to itself identically,  the $j$-th value $w_k^{(i,j)}(j)$ is $i-k+1$, the $i$-th value $w_k^{(i,j)}(i)$ is $i-k$, and the remaining entries in the one-line notation are arranged in increasing order from the $(j+1)$-st position to the $(i-1)$-st position.

\end{remark}

Recalling that the Schubert polynomial $\mathfrak{S}_w$ is identified with the Schubert class $\sigma_w$ under the isomorphism \eqref{eq: Borel}, the following is now straightforward.

\begin{corollary} \label{corollary:SchubertRelations}
Let box $(i,j)$ be a corner of $h$, and consider the map~\eqref{eq: restriction for Schubert}. 
The set of linear equations given by 
\begin{equation*}
\sum_{k=1}^{i-j} (-1)^{k-1} x_1^{i_1} \cdots x_{j-1}^{i_{j-1}} \cdot \overline{\sigma_{w_k^{(i,j)}} \cdot} x_{j+1}^{i_{j+1}} \cdots x_n^{i_n} 
= 0, 
\end{equation*} 
as $\mathbf{i} := (i_1,\ldots,i_{j-1},i_{j+1},\ldots,i_n) \in \Z^{n-1}_{\geq 0}$ varies over all  $\mathbf{i}$ such that 
$0 \leq i_m \leq h(m)-m$ for all $m$ with $1 \leq m \leq n, m \neq j$, form a basis for the kernel of~\eqref{eq: restriction for Schubert}.
\end{corollary} 

\begin{proof} 
This follows immediately from Theorem~\ref{theorem: alternating Schuberts} and Theorem~\ref{theorem: relation on Schuberts}.  
\end{proof} 

The next formula is equivalent to Monk's formula which is well-known on Schubert calculus.

\begin{theorem} [{Monk's formula \cite{Monk}, see also \cite[p.180--181]{Fulton}}] \label{theorem:Monk}
Let $\{\sigma_{w} \}_{w\in S_n}$ be the Schubert classes in the cohomology of the flag variety. 
Then we have
\begin{equation*} 
x_r \cdot \sigma_{w}=\sum_{w'} \sigma_{w'}-\sum_{w''} \sigma_{w''}
\end{equation*}
where the first sum is over those $w'$ obtained from $w$ by interchanging the values of $w$ in positions $r$ and $q$ for those $r<q$ with $w(r)<w(q)$, and $w(i)$ is not in the interval $(w(r),w(q))$ for any $i$ in the interval $(r,q)$, and the second sum is over those $w''$ obtained from $w$ by interchanging the values of $w$ in positions $r$ and $p$ for those $p<r$ with $w(p)<w(r)$, and $w(i)$ is not in the interval $(w(p),w(r))$ for any $i$ in the interval $(p,r)$.
\end{theorem}

The above arguments imply that a basis for the set of linear relations on Schubert classes on a given $\Hess(\mathsf{N},h)$ can be obtained by a step-by-step procedure as follows. 
First, choose a sequence of Hessenberg functions $h_0, h_1, \ldots, h_N$ such that $h_0=(n,n,\ldots,n)$ corresponds to the full flag variety, $h_N=h$ is the given Hessenberg function, and $h_{k+1}$ is obtained from $h_k$ by removing a corner box as in the discussion above. Second, at each step, 
use Corollary~\ref{corollary:SchubertRelations} and Theorem~\ref{theorem:Monk}
to obtain a basis of the linear relations satisfied by 
the Schubert classes $\overline{\sigma_{w}}$
in $H^*(\Hess(\mathsf{N},h_{k+1}))$ (but not satisfied in $H^*(\Hess(\mathsf{N},h_k))$). Taking the union of all such relations obtained at each step, we obtain a basis for the set of linear relations satisfied by Schubert classes in $H^*(\Hess(\mathsf{N},h=h_N))$. 

We give a small worked example. 

\begin{example} \label{example:LinearRelationsSchubertClasses}
Let $n=4$ and $h=(2,4,4,4)$. In this case, we can obtain $h$ from the Hessenberg function $h_0=(4,4,4,4)$ for the full flag variety in two steps: first we remove the box $(4,1)$ to obtain $h_1 = (3,4,4,4)$, and second, we remove the box $(3,1)$ to obtain $h_2 = h = (2,4,4,4)$. We consider both steps in sequence. At the first step, we have $i=4$ and $j=1$ and from Theorem~\ref{theorem: alternating Schuberts} we conclude that 
\[
f_{3,1} = \mathfrak{S}_{4123} - \mathfrak{S}_{3142} + \mathfrak{S}_{2341}
\]
and from 
Corollary~\ref{corollary:SchubertRelations}
we know we obtain $6$ linearly independent relations satisfied by Schubert classes in the cohomology $H^*(\Hess(\mathsf{N},h_1))$ by multiplying the above expression in Schubert classes by the six monomials 
$$\{1, x_2, x_3, x_2^2, x_2 x_3, x_2^2 x_3\}.$$ 
Using Theorem~\ref{theorem:Monk}, we obtain all linear relations among the Schubert classes in $H^*(\Hess(\mathsf{N},h_1))$ as follows:
\begin{align*} 
\overline{\sigma_{4123}}=\overline{\sigma_{3142}}-\overline{\sigma_{2341}}, \ \ \overline{\sigma_{4213}}=\overline{\sigma_{3412}}+2\overline{\sigma_{3241}}-\overline{\sigma_{2431}}, \ \ \overline{\sigma_{4132}}=\overline{\sigma_{3241}}, \ \ \overline{\sigma_{4312}}=\overline{\sigma_{3421}}, \ \ \overline{\sigma_{4231}}=0, \ \ 
\overline{\sigma_{4321}}=0. 
\end{align*}

At the second step, we have $i=3$ and $j=1$ so by Theorem~\ref{theorem: alternating Schuberts} we have 
\[
f_{2,1} = \mathfrak{S}_{3124} - \mathfrak{S}_{2314}
\]
and now again by 
Corollary~\ref{corollary:SchubertRelations}
we obtain another $6$ linearly independent relations by multiplying the above expression by the same six monomials $\{1, x_2, x_3, x_2^2, x_2 x_3, x_2^2 x_3\}$. 
Using Theorem~\ref{theorem:Monk} again, we obtain $6$ linearly independent relations.
Finally, the union of these sets of equations give us $12$ linearly independent relations which are satisfied by the (images of the) Schubert classes in $H^*(\Hess(\mathsf{N},h))$ 
as follows: 
\begin{align*} 
&\overline{\sigma_{3124}}=\overline{\sigma_{2314}}, \ \ 
\overline{\sigma_{4123}}=\frac{1}{2} \overline{\sigma_{2413}}, \ \ 
\overline{\sigma_{3214}}=\frac{1}{2} \overline{\sigma_{2413}}, \ \ 
\overline{\sigma_{3142}}=\frac{1}{2} \overline{\sigma_{2413}}+\overline{\sigma_{2341}}, \ \ 
\overline{\sigma_{4213}}=0, \ \ 
\overline{\sigma_{3241}}=\frac{1}{2} \overline{\sigma_{2431}}, \\ 
&\overline{\sigma_{4132}}=\frac{1}{2} \overline{\sigma_{2431}}, \ \ 
\overline{\sigma_{3412}}=0, \ \ 
\overline{\sigma_{4312}}=0, \ \ 
\overline{\sigma_{3421}}=0, \ \ 
\overline{\sigma_{4231}}=0, \ \ 
\overline{\sigma_{4321}}=0. 
\end{align*}
In particular, we can answer Question~\ref{conjecture:HaradaTymoczko} in the affirmative for $h_1 = (3,4,4,4)$ and $h=h_2 = (2,4,4,4)$.
\end{example}

\begin{remark}
As in Example~\ref{example:LinearRelationsSchubertClasses}, we can answer Question~\ref{conjecture:HaradaTymoczko} in the affirmative for all Hessenberg functions in the case $n=4$.
\end{remark}

%%%%%%%%%%%%%%%%
\section{Further directions: a proposal for a definition of Hessenberg Schubert polynomials}\label{sec: HS polynomials} 
%%%%%%%%%%%%%%%%%

In this last section, we propose a new research program: the study of Hessenberg Schubert polynomials, in the setting of regular nilpotent Hessenberg varieties and their cohomology rings. 
 Specifically, we propose a definition of a ``Hessenberg Schubert polynomial'', which generalizes the classical Schubert polynomials in the special case of $\Flags(\C^n)$. 
In this section we specify coefficients in cohomology.

To explain our definition, we need to recall some facts about the classical case (for a reference, see \cite{Fulton}). First, as mentioned in Remark~\ref{remark:CohomologyFlag} the cohomology ring $H^*(\Flags(\C^n);\Z)$ has the well-known Borel presentation 
\begin{equation}\label{eq: pi} 
\pi: \Z[x_1,\ldots, x_n] \to \Z[x_1,\ldots,x_n]/\langle e_1,\ldots, e_n\rangle \cong H^*(\Flags(\C^n);\Z)
\end{equation}
where, for each $i$, the $e_i$ denotes the $i$-th elementary symmetric polynomial. Second, in the cohomology of the flag variety, there exists a natural additive basis of Schubert classes $\{\sigma_w\}_{w \in S_n}$ where each $\sigma_w$ is the Poincar\'e dual to the (opposite) Schubert variety $\overline{B_{-} wB} =: \overline{\Omega_w^{\circ}}$. Third, classical Schubert calculus concerns the computation of the structure constants of the Schubert classes in $H^*(\Flags(\C^n);\Z)$, i.e. the $c^u_{wv}$ in the equation
\begin{equation}\label{eq: Schubert str const}
\sigma_w \cdot \sigma_v = \sum_{u \in S_n} c^u_{wv} \sigma_u.
\end{equation}
The structure constants $c^u_{wv}$ record intersection numbers for Schubert varieties, as studied in classical Schubert calculus.
Fourth, there exist certain polynomials $\mathfrak{S}_w \in \Z[x_1,\ldots,x_n]$ for $w \in S_n$ called the \textbf{Schubert polynomials}, with many good properties, two of which are that they map to the Schubert classes under the projection $\pi$ in~\eqref{eq: pi}, i.e. $\pi(\mathfrak{S}_w) = \sigma_w$ for all $w \in \S_n$, 
and in addition, the Schubert polynomials satisfy a stability property with respect to $n$.
Concretely, under the natural inclusion $i: S_n \hookrightarrow S_{n+1}$, the Schubert polynomial $\mathfrak{S}_w$ coincides with the Schubert polynomial $\mathfrak{S}_{i(w)}$, which allows us to define the Schubert polynomial $\mathfrak{S}_w$ for a permutation $w \in S_{\infty}:=\bigcup_{n \leq 1} S_n$.
It is known that these Schubert polynomials $\{\mathfrak{S}_w \mid w \in S_{\infty} \}$ form a basis for $\Z[x_1,x_2,\ldots]$ (with infinitely many variables) and the structure constants for the $\mathfrak{S}_w$'s match those for the $\sigma_w$'s, i.e., 
\[
\mathfrak{S}_w \cdot \mathfrak{S}_v = \sum_{u \in \S_\infty } c^u_{wv} \mathfrak{S}_u  \ \ \textup{ in } \Z[x_1,x_2,\cdots]
\]
where the constants $c^u_{wv}$ for $w,v,u \in S_n$ are those that appear in~\eqref{eq: Schubert str const} (but the multiplication now takes place in $\Z[x_1,x_2,\cdots]$). These Schubert polynomials can be defined using the well-known divided difference operators. For these reasons, among others, Schubert polynomials (and their analogues for equivariant cohomology, $K$-theory, etc.) are an integral ingredient in the study of Schubert calculus. It thus seems natural to ask whether there exist analogues of Schubert polynomials in the more general setting of Hessenberg varieties. Our proposed definition of such polynomials, made precise in Definition~\ref{definition: Hess Schub poly} below, is motivated by the following fact. 
Let $w \in \S_n$. Then for all $I = (i_1, \ldots, i_n) \in \Z^{n}_{\geq 0}$ satisfying $0 \leq i_m \leq n-m$ for all $1 \leq m \leq n$ there exists integers $a^w_I \in \Z$ such that 
\begin{equation}\label{eq: change of basis flag}
\mathfrak{S}_w = 
\sum_{\substack{I=(i_1, \ldots, i_n) \\ 0 \leq i_m \leq n-m, \, \forall m}}
a^w_I x_1^{i_1} x_2^{i_2} \cdots x_n^{i_n}.
\end{equation}
Note that the constants $a^w_I$ are in fact determined by the following equality in $H^*(\Flags(\C^n);\Z)$ 
\[
\sigma_w =
\sum_{\substack{I=(i_1, \ldots, i_n) \\ 0 \leq i_m \leq n-m, \, \forall m}}
a^w_I x_1^{i_1} x_2^{i_2} \cdots x_n^{i_n}
\]
where the RHS is interpreted in the quotient ring $H^*(\Flags(\C^n);\Z) \cong \Z[x_1,\ldots,x_n]/\langle e_1,\ldots, e_n\rangle$, since both the Schubert classes $\{\sigma_w\}_{w \in \S_n}$ and the monomials $\{x_1^{i_1} \cdots x_n^{i_n} \, \mid \, 0 \leq i_m \leq n-m \}$ form additive bases of $H^*(\Flags(\C^n);\Z)$. Put another way, the constants $a^w_I$ are the coefficients in the change-of-basis matrix relating the monomial basis $\{x_1^{i_1} \cdots x_n^{i_n} \, \mid \, 0 \leq i_m \leq n-m \}$ to the Schubert basis $\{\sigma_w\}_{w \in \S_n}$. It is well-known that these coefficients $a^w_I$ are non-negative integers 
\cite[(4.17)]{Macdonald}.

The essential idea of our Definition~\ref{definition: Hess Schub poly} is to define Hessenberg Schubert polynomials using the constants in an analogous change-of-basis matrix for two additive bases of $H^*(\Hess(\mathsf{N},h);\Q)$, i.e., to view the equality in~\eqref{eq: change of basis flag} as a \emph{definition} of the Schubert polynomials, with appropriate choices of bases on both the LHS and RHS. Henceforth we work with $\Q$ coefficients for the cohomology rings since our results in the previous sections hold over $\Q$. One of the additive bases we will use is the set of monomials $\{x_1^{i_1} x_2^{i_2} \cdots x_n^{i_n} \, \mid \, 0 \leq i_m \leq h(m)-m \}$, shown to be a basis of $H^*(\Hess(\mathsf{N},h);\Q)$ in 
Corollary~\ref{cor: basis}
of Section~\ref{section:monomial basis}. For the appropriate analogue of the basis of Schubert classes in the Hessenberg setting, we need some results of the fifth author \cite{Tymoczko}, as we now explain. 

Let $\Omega_w^{\circ}$ denote the (opposite) Schubert cell corresponding to $w \in \S_n$. 
In what follows, we take a regular nilpotent matrix $\mathsf{N}$ as the following form
\begin{equation*}
\mathsf{N}=\begin{pmatrix}
0 &    &  &  & \\
1 & 0 &   & \\
  & 1  & \ddots &  & \\
  &    & \ddots  & \ddots &  \\
  &    &  & 1 & 0 \\ 
\end{pmatrix}, 
\end{equation*}
which is the conjugate of the regular nilpotent matrix in Jordan canonical form by the longest element $w_0 \in S_n$. 
The fifth author proves in \cite{Tymoczko} that the intersections
\[
\{ \Hess(\mathsf{N},h) \cap \Omega_w^{\circ} \, \mid \, w \in \S_n \textup{ such that } \Hess(\mathsf{N},h) \cap \Omega_w^{\circ} \neq \emptyset \}
\]
form an affine paving of $\Hess(\mathsf{N},h)$. This implies in particular that the homology classes corresponding to their closures 
\[
\Theta_w^{h} := [ \overline{\Hess(\mathsf{N},h) \cap \Omega_w^{\circ}}] \in H_*(\Hess(\mathsf{N},h);\Z)
\]
form an additive basis of the homology $H_*(\Hess(\mathsf{N},h);\Z)$. We now use the fact that, although $\Hess(\mathsf{N},h)$ is in general singular, its cohomology ring $H^*(\Hess(\mathsf{N},h);\Q)$ with \textit{rational} coefficients is nevertheless a Poincar\'e duality algebra \cite[Proposition 10.6]{AHHM}, so it still has non-degeneracy properties similar to the cohomology rings of compact smooth oriented manifolds. 
The fact that $H^*(\Hess(\mathsf{N},h);\Q)$ is a Poincar\'e duality algebra follows from an algebraic argument, but we can also take a geometrically natural generator $\beta_h$ of the top degree cohomology $H^{top}(\Hess(\mathsf{N},h);\Q)$ as follows \cite[Section~3]{EHNT19}:
\begin{equation} \label{eq:TopGenerator}
\beta_h:=\frac{1}{|S_h|}\prod_{j=1}^{n-1} \prod_{i=j+1}^{h(j)} (x_j-x_i)
\end{equation}
where $S_h$ is the Young subgroup of $S_n$ associated with $h$ defined as follows. 
Let $\ell$ be the number of integers $i$ with $1 \leq i \leq n$ satisfying $h(m)=m$ and let 
$\{k_i \, \mid \, 1 \leq i \leq \ell\}$ be precisely these integers, i.e., $h(k_i)=k_i$ for all $1\leq i \leq \ell$, and assume $k_1 < k_2< \cdots < k_\ell=n$.
Then, $S_h$ is defined to be the product of smaller symmetric groups $S_{k_1} \times S_{k_2-k_1} \times \cdots S_{k_{\ell}-k_{\ell-1}}$.
Then, by the non-degeneracy assumption of Poincar\'e duality algebras, we may define the Poincar\'e dual class in $H^*(\Hess(\mathsf{N},h);\Q)$ of $\Theta_w^h$ and thereby obtain
\[
\sigma_w^h := \mathrm{PD}(\Theta_w^h) \in H^*(\Hess(\mathsf{N},h);\Q)
\]
where $\mathrm{PD}$ denotes the Poincar\'e dual with respect to the choice of generator made in \eqref{eq:TopGenerator}. Specifically, $\mathrm{PD}$ is the isomorphism $H_{2d-2k}(\Hess(\mathsf{N},h);\Q) \cong H^{2k}(\Hess(\mathsf{N},h);\Q)$ via the non-degeneracy of the pairing
\begin{equation} \label{eq:PDiso}
H^{2k}(\Hess(\mathsf{N},h);\Q) \times H^{2d-2k}(\Hess(\mathsf{N},h);\Q) \to H^{2d}(\Hess(\mathsf{N},h);\Q) \xrightarrow[\cong]{\int} \Q
\end{equation} 
for $0 \leq k \leq d:=\sum_{j=1}^n (h(j)-j)$, 
where the isomorphism $\int: H^{2d}(\Hess(\mathsf{N},h);\Q) \xrightarrow{\cong} \Q$ sends the generator $\beta_h$ in \eqref{eq:TopGenerator} to $1$.
From the above discussion it then follows that the set 
\[
\{\sigma_w^h \, \mid \, w \in \S_n \textup{ such that } \Hess(\mathsf{N},h) \cap \Omega_w^{\circ} \neq \emptyset \}
\]
is an additive basis of $H^*(\Hess(\mathsf{N},h);\Q)$. We take these to be the appropriate analogues of the Schubert classes in the classical case. In Section~\ref{section:monomial basis} we derived a different basis for $H^*(\Hess(\mathsf{N},h);\Q)$, consisting of monomials. It follows from basic linear algebra that there exist constants $a^w_I \in \Q$ for $w$ such that $\Hess(\mathsf{N},h) \cap \Omega_w^{\circ} \neq \emptyset$ and for $I=(i_1,\ldots, i_n)$ with $0 \leq i_m \leq h(m)-m$ for all $m$ such that 
\begin{equation}\label{eq: def proto for Hess Schub poly}
\sigma_w^h =\sum_{\substack{I=(i_1, \ldots, i_n) \\ 0 \leq i_m \leq h(m)-m, \, \forall m}}
a^w_I x_1^{i_1} x_2^{i_2} \cdots x_n^{i_n}
\end{equation}
where the equality is interpreted in $H^*(\Hess(\mathsf{N},h);\Q)$. We can now state our definition.

\begin{definition}\label{definition: Hess Schub poly}
Let $h: [n] \to [n]$ be a Hessenberg function. Let $w \in \S_n$ such that $\Hess(\mathsf{N},h) \cap \Omega_w^{\circ} \neq \emptyset$. We define the Hessenberg Schubert polynomial $\mathfrak{S}^h_w$ as 
\begin{equation}\label{eq: def Hess Schub poly}
\mathfrak{S}^h_w := 
\sum_{\substack{I=(i_1, \ldots, i_n) \\ 0 \leq i_m \leq h(m)-m, \, \forall m}}
a^w_I x_1^{i_1} x_2^{i_2} \cdots x_n^{i_n}
\end{equation}
where the RHS is considered as an element in $\Q[x_1,\ldots,x_n]$ and the constants $a^w_I \in \Q$ are defined by the equalities~\eqref{eq: def proto for Hess Schub poly}. 
\end{definition}

It is not hard to see that Definition~\ref{definition: Hess Schub poly} generalizes the usual Schubert polynomials, i.e., 
$\mathfrak{S}^h_w = \mathfrak{S}_w$ for all $w \in \S_n$ if $h=(n,n,\ldots,n)$.
This is because when $h=(n,n,\ldots,n)$, the $\beta_h$ in \eqref{eq:TopGenerator} is the Poincar\'e dual of a point in the flag variety (cf. \cite[Equation~(3.5)]{EHNT19}). 
Therefore, the isomorphism $\int$ in \eqref{eq:PDiso} coincides with usual evaluation map sending the Poincar\'e dual of a point to $1 \in \Z$, and in this case $\sigma_w^h$ are equal to the usual Schubert classes $\sigma_w$.

In light of the above discussion, we believe that these Hessenberg Schubert polynomials are a natural generalization of the classical Schubert polynomials and should enjoy similar properties as in the classical case. For instance, we may ask: 

\begin{itemize} 
\item What properties do the coefficients $a^w_I$ satisfy and can they be interpreted geometrically?
\item Do the $\mathfrak{S}^h_w$ satisfy ``stability'' conditions as in the classical case? 
\end{itemize} 
Here, the ``stability'' in the second question is intended in the following sense. Let $H_n$ be the set of Hessenberg functions $h: [n] \to [n]$. Then, we can consider the natural inclusion $j: H_n \rightarrow H_{n+1}$ so that $j(h): [n+1] \to [n+1]$ is the Hessenberg function defined by $(j(h))(i)=h(i)$ for $1 \leq i \leq n$ and $(j(h))(n+1)=n+1$. We also consider the natural inclusion $i: S_n \hookrightarrow S_{n+1}$. Then, we may ask: does the Hessenberg Schubert polynomial $\mathfrak{S}^h_w$ coincide with the Hessenberg Schubert polynomial $\mathfrak{S}^{j(h)}_{i(w)}$?

 We intend to explore these and related questions in future work.

\end{document}